\documentclass[11pt]{article}
\usepackage{amssymb,amsmath,amsthm}
\usepackage{geometry}
\usepackage{multicol}
\usepackage{epsfig,graphicx}
\usepackage{tikz}
\usetikzlibrary{automata,positioning}
\usepackage{enumerate}
\usepackage{subcaption}
\usepackage{diagbox}
\usepackage{url}
\usepackage{appendix}

\renewcommand{\pmod}[1]{{\ifmmode\text{\rm\ (mod~$#1$)}\else\discretionary{}{}{\hbox{ }}\rm(mod~$#1$)\fi}}

\newtheorem{theorem}{Theorem}[section]
\newtheorem{prop}[theorem]{Proposition}
\newtheorem{lemma}[theorem]{Lemma}
\newtheorem{cor}[theorem]{Corollary}
\newtheorem{conjecture}[theorem]{Conjecture}

\theoremstyle{definition}

\newtheorem{defn}[theorem]{Definition}

\numberwithin{equation}{section}
\numberwithin{table}{section}
\numberwithin{figure}{section}

\title{Difference Necklaces}
\author{Ethan P.\ White\footnote{The first author is supported by funding from NSERC and the Killam Trusts.}, Richard K.\ Guy and Renate Scheidler\footnote{The third author is supported by an NSERC Discovery Grant.}}

\begin{document}
\maketitle
{\let\thefootnote\relax\footnote{2020 \emph{Mathematics Subject Classification}. 11B37, 05C30, 05C45, 11Y55}}


\begin{abstract} 
An $(a,b)$-difference necklace of length $n$ is a circular arrangement of the integers $0, 1, 2, \ldots , n-1$ such that any two neighbours have absolute difference $a$ or $b$. We prove that, subject to certain conditions on $a$ and $b$, such arrangements exist, and provide recurrence relations for the number of $(a,b)$-difference necklaces for $( a, b ) = ( 1, 2 )$, $( 1, 3 )$, $( 2, 3 )$ and $( 1, 4 )$. Using techniques similar to those employed for enumerating Hamiltonian cycles in certain families of graphs, we obtain these explicit recurrence relations and prove that the number of $(a,b)$-difference necklaces of length $n$ satisfies a linear recurrence relation for all permissible values $a$ and $b$. Our methods generalize to necklaces where an arbitrary number of differences is allowed. 
\end{abstract}

\noindent \textbf{Keywords.} Difference necklace, recurrence relation, Hamiltonian cycle, weighted digraph, transfer matrix method.


\section{Introduction}

\subsection{Background and Motivation} 

It is natural to ask whether it is possible to order the integers $1, 2, \ldots , n$ in such a way that any two neighbours are subject to certain conditions. In case such arrangements are known to exist for any given length $n$, an obvious follow-up task is to count them. Although seemingly recreational in nature, these types of questions tend to be extremely difficult and bear relationships to important problems in combinatorics, graph theory and number theory.

A \emph{chain} is an ordering of $1, 2, \ldots , n$ that imposes an arithmetic  restriction on any two adjacent numbers; a circular chain, where the first and last term are also considered neighbours, is referred to as a \emph{necklace} and its elements as \emph{beads}. In \cite{BG}, Berlekamp and the second author investigated necklaces where any two adjacent beads sum to a particular type, such as a square, cube or triangular number. They illustrated how a search for such chains and necklaces can be facilitated by considering paths of billiard balls on a rectangular or other polygonal billiard table. Using this technique, they gave necessary and sufficient conditions for the existence of chains where neighbours sum to a Fibonacci number or a Lucas number.

The problem of finding a \emph{square sum chain} of length 15, i.e.\ a chain for which the sum of any two adjacent terms is a perfect square, was posed as Puzzle 4 in~\cite{Y}. The chain
\[ 9,7,2,14,11,5,4,12,13,3,6,10,15,1,8 \] 
is a solution that can be verified to be unique by inspecting the graph on the vertices $1,2,\ldots,15$ with an edge joining two vertices if and only if their sum is a square. Generalizing this idea, the sequence A090461 in the \emph{Online Encyclopedia of Integer Sequences}~\cite{NS} describes the integers $n$ for which there exists a square sum chain of length $n$, and the sequences A090460 and A071984 count the number of essentially different square sum necklaces and square sum chains, respectively. Other related sequences include A108658, A272259 and A107929. Surprisingly, the existence question for square sum chains and necklaces was only recently settled completely by R. Gerbicz in his Mersenne Forum blog post \cite{MS}. Deploying a combination of construction and computation, he proved that square sum necklaces of length $n$ exist for all $n \ge 32$ and square sum chains for all lengths $n \ge 25$. Very little seems to be known beyond the cases discussed here. Some cube sum chains and necklaces can be found on the puzzle site \cite{PP}; their minimal lengths are 305 for a cube chain and 473 for a cube necklace. The question of arbitrary power sum chains was posed and discussed on Math Overflow \cite{MO} in 2015; see also sequence A304120 in \cite{NS}.

An entirely different picture emerges when considering \emph{difference chains/necklaces} where sums are replaced by differences (up to sign). Since 1 is a square, the integers $1, 2, \ldots , n$ in this order always form a square difference chain, which is a square difference necklace when $n-1$ is a square. A version of this problem that rules out 1 as a permissible square was briefly considered in \cite{MD}. If the value 1 is allowed, square difference necklaces seem to be far more abundant than their sum counterparts, even if only two square values (and their negatives) are allowed. For example, difference necklaces of any length $\ge 7$ exist in which adjacent terms differ in absolute value by~1 or~$4$ (see Section~\ref{count14}). They also exist for any length beyond a sufficiently large threshold when~4 is replaced by any even square. Necklaces with two possible neighbour differences (and their negatives) also exist for many other pairs of values beyond squares (see Theorem~\ref{exist}). This motivated us to investigate these types of difference necklaces in more detail. 

Our objects of interest are \emph{$(a,b)$-difference necklaces}, i.e.\ necklaces for which the absolute difference of any two adjacent beads takes on one of two possible values $a$ or $b$. Consider a grid on $n$ vertices, where two vertices are adjacent if and only if their difference is $\pm a$ or $\pm b$; see Figure~\ref{15graph} for example. Then the $(a,b)$-difference necklaces of length~$n$ are precisely the Hamiltonian cycles in this graph. The structure of this graph is similar to that of a grid graph on $n$ vertices, especially if one of the difference values is~1. The $m \times k$ \emph{$($rectangular$)$ grid graph} is the Cartesian product of two paths of respective lengths~$m$ and~$k$. Counting Hamiltonian cycles in grid graphs is a difficult problem that has received significant attention and has to date only been solved in certain special cases. Solutions for $m \leq 6$ can be found in \cite{KD} and the sources cited therein. Pettersson~\cite{P} used dynamic programming to obtain counts for $m = k \le 26$. Stoyan and Strehl~\cite{SS} encoded Hamiltonian cycles in grid graphs as words in a regular language and constructed a finite automaton that recognizes this language. They determined the generating functions for the number of Hamiltonian cycles for $m \leq 8$ and the number of cycles of length in the thousands for $9 \leq m \leq 12$. 

Counting and enumerating objects that are closely related to  difference chains and necklaces also arise in contexts other than enumerating Hamiltonian cycles; we mention only a few here. Flajolet and Sedgewick's kangaroo jumping problem \cite[p.\ 373]{FS} can be expressed as finding the number of chains starting at 1 and ending at $n+1$ that allow neighbour differences of $\pm 1$ and 2. The solution represents the counting function for many objects, as shown by the extensive entry for A000930 in \cite{NS}. Flajolet \emph{et al}'s results on threshold probabilities for the robustness of a random graph \cite{PF} utilize the enumeration of \emph{avoiding permutations} which are difference chains that allow any neighbour difference other than $\pm 1$. 

In Section~\ref{S:exist}, we prove that under some obvious necessary conditions on $a$ and $b$, $(a,b)$-difference necklaces of sufficiently large length always exist when $2a \le b$. Recurrence relations for the counts of $(a,b)$-difference necklaces with $(a,b) = (1,2), (1,3), (2,3)$ and $(1,4)$ are given in Section~\ref{S:count}. Our main result appears in Section~\ref{S:general} and establishes that the number of $(a,b)$-difference necklaces satisfies a linear homogeneous recurrence relation for any permissible values $a, b$. Our proof technique is similar to the Kwong-Rogers approach in~\cite{KD} which uses a variant of the transfer matrix method to enumerate Hamiltonian cycles in the grid graphs $P_4 \times P_k$ and $P_5 \times P_k$ where $P_j$ denotes a path of length~$j$. We close with concluding remarks in Section~\ref{S:conclude} and provide some computational data for counts of $(a,b)$-necklaces for various pairs $(a,b)$ in an appendix at the end.
\subsection{Notation} \label{notation}

Throughout, let $n, a, b$ be positive integers with $n \ge 3$ and $a < b$. For ease of notation in subsequent proofs, we consider permutations of the integers $0, 1, \ldots , n-1$ rather than $1, 2, \ldots , n$; this downshift by 1 does not affect any of our results. Moreover, since we only consider difference necklaces here, we will henceforth refer to them simply as necklaces.

\begin{defn}
An \emph{$(a,b)$-necklace of length $n$} is a circular arrangement of the integers $0, 1, \ldots, n-1$ such that any two adjacent beads have difference $\pm a$ or $\pm b$. The number of distinct $(a,b)$-necklaces of length $n$ is denoted $N_{a,b}(n)$.
\end{defn}

Here, we consider two $(a,b)$-necklaces of length $n$ to be distinct if one cannot be obtained from the other by rotation or by reflection on some axis through the centre of the circular arrangement. 
A $(4,7)$-necklace of length 11 is depicted in Figure~\ref{47}.

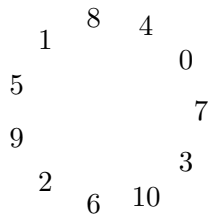
\begin{figure}[h!]    
\begin{center}
\begin{tikzpicture}
{
\draw (1*360/11 :1.25 cm) node{0}; 
\draw (2*360/11 :1.25 cm) node{4}; 
\draw (3*360/11 :1.25 cm) node{8}; 
\draw (4*360/11 :1.25 cm) node{1};
\draw (5*360/11 :1.25 cm) node{5};
\draw (6*360/11 :1.25 cm) node{9};
\draw (7*360/11 :1.25 cm) node{2};
\draw (8*360/11 :1.25 cm) node{6};
\draw (9*360/11 :1.25 cm) node{10};
\draw (10*360/11 :1.25 cm) node{3};
\draw (11*360/11 :1.25 cm) node{7};}
\end{tikzpicture}

\end{center}
\caption{A $(4,7)$-necklace of length 11. Any two neighbours differ by $\pm 4$ or $\pm 7$.}
\label{47}
\end{figure}

Let $G_{a,b}(n)$ be the graph with vertex set $\{0, 1, \ldots n-1 \}$ where two vertices $x,y$ are adjacent if and only if $|x-y| \in \{a,b\}$. Then the distinct $(a,b)$-necklaces of length $n$ are in one-to-one correspondence with the Hamiltonian cycles in $G_{a,b}(n)$. Figure~\ref{15graph} shows the graph $G_{1,5}(18)$.

\begin{figure}[h!]
\begin{center}
\begin{tikzpicture}[scale=1]

\node[draw][circle](1) at (0,4) {0};
\node[draw][circle](2) at (0,3) {1};
\node[draw][circle](3) at (0,2) {2};
\node[draw][circle](4) at (0,1) {3};
\node[draw][circle](5) at (0,0) {4};
\node[draw][circle](6) at (1.7,4) {5};
\node[draw][circle](7) at (1.7,3) {6};
\node[draw][circle](8) at (1.7,2) {7};
\node[draw][circle](9) at (1.7,1) {8};
\node[draw][circle](10) at (1.7,0) {9};;
\node[draw][circle, inner sep=2pt](11) at (3.4,4) {10};
\node[draw][circle, inner sep=2pt](12) at (3.4,3) {11};
\node[draw][circle, inner sep=2pt](13) at (3.4,2) {12};
\node[draw][circle, inner sep=2pt](14) at (3.4,1) {13};
\node[draw][circle, inner sep=2pt](15) at (3.4,0) {14};
\node[draw][circle, inner sep=2pt](16) at (5.1,4) {15};
\node[draw][circle, inner sep=2pt](17) at (5.1,3) {16};
\node[draw][circle, inner sep=2pt](18) at (5.1,2) {17};

\draw (1)--(2)--(3)--(4)--(5)--(6)--(7)--(8)--(9)--(10)
(10)--(11)--(12)--(13)--(14)--(15)--(16)--(17)--(18) (1)--(6)--(11)--(16) (2)--(7)--(12)--(17)
(3)--(8)--(13)--(18) (4)--(9)--(14) (5)--(10)--(15);
\end{tikzpicture}
\end{center}
\caption{The graph $G_{1,5}(18)$. Any two adjacent vertices differ by $\pm 1$ or $\pm 5$.} \label{15graph}
\end{figure}
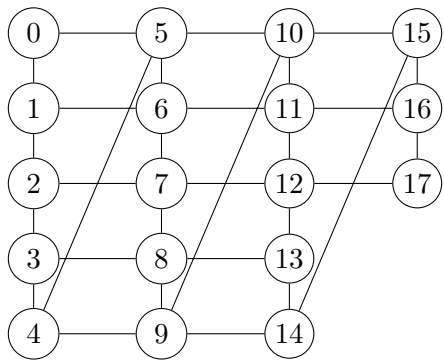

When drawing $G_{a,b}(n)$, we will always arrange the vertices in columns of length $b$ in ascending order to form a ragged rectangular grid as illustrated in Figure~\ref{15graph}. When $a = 1$, this vertex arrangement is very similar to the grid graph $P_{\lceil n/b \rceil} \times P_b$, except that the number of vertices in the rightmost column may be less than $b$.

Note that $G_{a,b}(n)$ is a subgraph of the circulant graph on $n$ vertices with jumps $a, b$ (the latter allows additional neighbour differences $\pm(n-a)$ and $\pm(n-b)$).


\section{Existence of $(a,b)$-Necklaces} \label{S:exist}

All beads in an $(a,b)$-necklace must have the same remainder modulo $\gcd(a,b)$, so such necklaces do not exist when $a$ and $b$ have a non-trivial common factor. In fact, in this case, $G_{a,b}(n)$ is disconnected. We therefore assume henceforth that $a$ and $b$ are coprime. Moreover, if $a$ and $b$ are both odd, then $G_{a,b}(n)$ is bipartite since there are no edges between vertices of the same parity. Consequently, there are no $(a,b)$-necklaces of odd length in this case.

On the other hand, if $a$ and $b$ are coprime, then $G_{a,b}(a+b)$ is the circulant graph on jumps $a, b$ which is 2-regular and connected since gcd$(a+b,a,b)=1$. Hence it is a cycle, and this cycle represents the unique $(a,b)$-necklace of length $a+b$. No $(a,b)$-necklaces of length less than $a+b$ exist, since the only choices for neighbours of bead $b-1$ are $b-a-1$, $b+a-1$ and $2b-1$, and the last two of these three numbers exceeds $b+a-2$. 

In this section, we prove the existence of $(a,b)$-necklaces of any sufficiently large length~$n$ subject to the aforementioned conditions on $a, b$ and the additional restriction that \mbox{$2a \le b$}. For $a=1$, this inequality always holds, and the proof constructs an explicit Hamiltonian cycle in $G_{a,b}(n)$. For $a \ge 2$, the proof proceeds in three stages. First, we build an $(a,b)$-necklace of length $3a+b$. Next, we illustrate how two $(a,b)$-necklaces of respective lengths $m$ and $n$ can be ``glued together'' to form an $(a,b)$-necklace of length $m+n$. Finally, the existence of $(a,b)$-necklaces of respective lengths $a+b$ and $3a+b$ allows the conclusion that there are $(a,b)$-necklaces of any sufficiently large length.

We begin with the case $a = 1$ and exploit the similarity of $G_{1,b}(n)$ to the grid graph $P_{\lceil n/b \rceil} \times P_b$ to construct an explicit Hamiltonian cycle. Recall that $P_m \times P_k$ is Hamiltonian if and only if $km$ is even, and a Hamiltonian cycle can be traced explicitly by a ``snaking'' pattern. A similar technique can be employed in our context.

\begin{lemma}\label{a=1} 
If $b\geq 2$ is even, then there exists a $(1,b)$-necklace of length $n$ for all $n \geq 2b$. If $b \geq 3$ is odd, then there exists a $(1,b)$-necklace of length $n$ for all even $n \geq 2b$. 
\end{lemma}
\begin{proof}
Recall that $(1,b)$-necklaces only exist for even lengths when $b$ is odd. We present snake patterns, depending on the parity of $b$ and $n \pmod{b}$, that generate a Hamiltonian cycle in $G_{1,b}(n)$. Note that when $b$ is even, $n \pmod{b}$ has the same parity as~$n$. This results in four cases, yielding four patterns: (A) $b$ and $n$ even, (B) $b$ even and $n$ odd, (C) $b$ odd and $n \pmod b$ even and (D) $b$ and $n \pmod b$  odd. Using our aforementioned graphical vertex arrangement for $G_{1,b}(n)$, we illustrate the idea in Figure~\ref{snakes}.

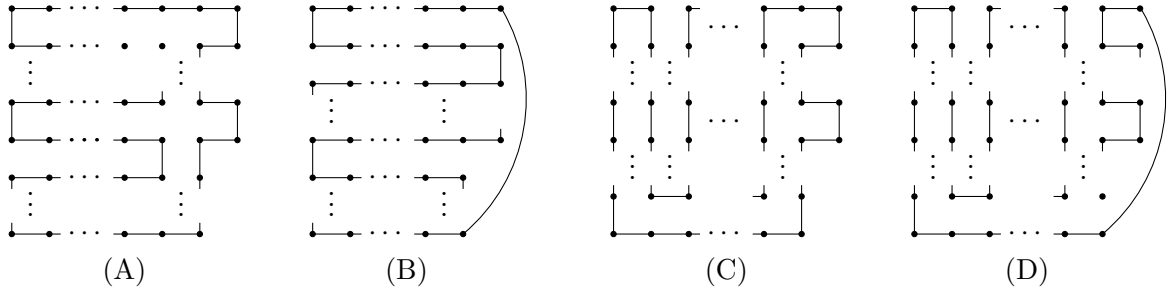
\begin{figure}[h!]
\begin{center}
\begin{tikzpicture}[scale=0.5]

\filldraw

(0,6) circle(2pt) 
(1,6) circle(2pt) 
(3,6) circle(2pt) 
(4,6) circle(2pt) 
(5,6) circle(2pt) 
(6,6) circle(2pt) 

(0.5,4.5) node{$\vdots$}
(4.5,4.5) node{$\vdots$}

(0,5) circle(2pt) 
(1,5) circle(2pt) 
(3,5) circle(2pt) 
(4,5) circle(2pt) 
(5,5) circle(2pt) 
(6,5) circle(2pt)

(2,6) node{$\cdots$}
(2,5) node{$\cdots$}

(0,3.5) circle(2pt) 
(1,3.5) circle(2pt) 
(3,3.5) circle(2pt) 
(4,3.5) circle(2pt) 
(5,3.5) circle(2pt) 
(6,3.5) circle(2pt)

(0,2.5) circle(2pt) 
(1,2.5) circle(2pt) 
(3,2.5) circle(2pt) 
(4,2.5) circle(2pt) 
(5,2.5) circle(2pt) 
(6,2.5) circle(2pt)

(2,3.5) node{$\cdots$}
(2,2.5) node{$\cdots$}

(0,1.5) circle(2pt) 
(1,1.5) circle(2pt) 
(3,1.5) circle(2pt) 
(4,1.5) circle(2pt) 
(5,1.5) circle(2pt)

(0.5,1) node{$\vdots$}
(4.5,1) node{$\vdots$}

(2,2.5) node{$\cdots$}
(2,2.5) node{$\cdots$}

(0,0) circle(2pt) 
(1,0) circle(2pt) 
(3,0) circle(2pt) 
(4,0) circle(2pt) 
(5,0) circle(2pt)

(2,1.5) node{$\cdots$}
(2,0) node{$\cdots$};

\draw

(3,-1) node{(A)}

(0,6)--(1.3,6)
(0,5)--(1.3,5)
(0,6)--(0,5)
(2.7,6)--(6,6)--(6,5)--(5,5)--(5,4.7)
%
(0,3.5)--(1.3,3.5)
(0,2.5)--(1.3,2.5)
(0,3.5)--(0,2.5)
(2.7,3.5)--(4,3.5)--(4,3.8)
(5,3.8)--(5,3.5)--(6,3.5)--(6,2.5)--(5,2.5)--(5,1.2)
(0,1.2)--(0,1.5)--(1.3,1.5)
(2.7,1.5)--(4,1.5)--(4,2.5)--(2.7,2.5)
(0,0.3)--(0,0)--(1.3,0)
(2.7,0)--(5,0)--(5,0.3);

 \begin{scope}[xshift=80mm]

\filldraw
(0,0) circle (2pt)
(1,0) circle (2pt)
(3,0) circle (2pt)
(4,0) circle (2pt)

(0,1.5) circle (2pt)
(1,1.5) circle (2pt)
(3,1.5) circle (2pt)
(4,1.5) circle (2pt)

(0,2.5) circle (2pt)
(1,2.5) circle (2pt)
(3,2.5) circle (2pt)
(4,2.5) circle (2pt)
(5,2.5) circle (2pt)

(0,4) circle (2pt)
(1,4) circle (2pt)
(3,4) circle (2pt)
(4,4) circle (2pt)
(5,4) circle (2pt)

(0,5) circle (2pt)
(1,5) circle (2pt)
(3,5) circle (2pt)
(4,5) circle (2pt)
(5,5) circle (2pt)

(0,6) circle (2pt)
(1,6) circle (2pt)
(3,6) circle (2pt)
(4,6) circle (2pt)
(5,6) circle (2pt)

(2,0) node{$\cdots$}
(2,1.5) node{$\cdots$}
(2,2.5) node{$\cdots$}
(2,4) node{$\cdots$}
(2,5) node{$\cdots$}
(2,6) node{$\cdots$}
(0.5,1) node{$\vdots$}
(3.5,1) node{$\vdots$}
(0.5,3.5) node{$\vdots$}
(3.5,3.5) node{$\vdots$};

\draw

(2.5,-1) node{(B)}
(4,0) arc(130:210:-4.7)

(0,0.3)--(0,0)--(1.3,0)
(2.7,0)--(4,0)
(5,6)--(2.7,6)
(1.3,6)--(0,6)--(0,5)--(1.3,5)
(2.7,5)--(5,5)--(5,4)--(2.7,4)
(1.3,4)--(0,4)--(0,3.7)
(5,2.8)--(5,2.5)--(2.7,2.5)
(1.3,2.5)--(0,2.5)--(0,1.5)--(1.3,1.5)
(2.7,1.5)--(4,1.5)--(4,1.2);

 \end{scope}
 
 \begin{scope}[xshift=160mm]
 
\filldraw

(0,0) circle(2pt)
(1,0) circle(2pt)
(2,0) circle(2pt)
(4,0) circle(2pt)
(5,0) circle(2pt)

(0,1) circle(2pt)
(1,1) circle(2pt)
(2,1) circle(2pt)
(4,1) circle(2pt)
(5,1) circle(2pt)

(0,2.5) circle(2pt)
(1,2.5) circle(2pt)
(2,2.5) circle(2pt)
(4,2.5) circle(2pt)
(5,2.5) circle(2pt)
(6,2.5) circle(2pt)

(0,3.5) circle(2pt)
(1,3.5) circle(2pt)
(2,3.5) circle(2pt)
(4,3.5) circle(2pt)
(5,3.5) circle(2pt)
(6,3.5) circle(2pt)

(0,5) circle(2pt)
(1,5) circle(2pt)
(2,5) circle(2pt)
(4,5) circle(2pt)
(5,5) circle(2pt)
(6,5) circle(2pt)

(0,6) circle(2pt)
(1,6) circle(2pt)
(2,6) circle(2pt)
(4,6) circle(2pt)
(5,6) circle(2pt)
(6,6) circle(2pt)

(0.5,4.5) node{$\vdots$}
(1.5,4.5) node{$\vdots$}
(4.5,4.5) node{$\vdots$}
(0.5,2) node{$\vdots$}
(1.5,2) node{$\vdots$}
(4.5,2) node{$\vdots$}

(3,3) node{$\cdots$}
(3,0) node{$\cdots$}
(3,5.5) node{$\cdots$};

\draw
(3,-1) node{(C)}

(4,4.7)--(4,6)--(6,6)--(6,5)--(5,5)--(5,4.7)
(5,3.8)--(5,3.5)--(6,3.5)--(6,2.5)--(5,2.5)--(5,2.2)

(5,1.3)--(5,0)--(3.7,0)
(2.3,0)--(0,0)--(0,1.3)
(1,1.3)--(1,1)--(2,1)--(2,1.3)
(0,4.7)--(0,6)--(1,6)--(1,4.7)
(2,4.7)--(2,6)--(2.3,6)
(3.7,1)--(4,1)--(4,1.3)

(0,2.2)--(0,3.8)
(1,2.2)--(1,3.8)
(2,2.2)--(2,3.8)
(4,2.2)--(4,3.8);

  \end{scope}
  
   \begin{scope}[xshift=240mm]
   
   \filldraw

(0,0) circle(2pt)
(1,0) circle(2pt)
(2,0) circle(2pt)
(4,0) circle(2pt)
(5,0) circle(2pt)

(0,1) circle(2pt)
(1,1) circle(2pt)
(2,1) circle(2pt)
(4,1) circle(2pt)
(5,1) circle(2pt)

(0,2.5) circle(2pt)
(1,2.5) circle(2pt)
(2,2.5) circle(2pt)
(4,2.5) circle(2pt)
(5,2.5) circle(2pt)
(6,2.5) circle(2pt)

(0,3.5) circle(2pt)
(1,3.5) circle(2pt)
(2,3.5) circle(2pt)
(4,3.5) circle(2pt)
(5,3.5) circle(2pt)
(6,3.5) circle(2pt)

(0,5) circle(2pt)
(1,5) circle(2pt)
(2,5) circle(2pt)
(4,5) circle(2pt)
(5,5) circle(2pt)
(6,5) circle(2pt)

(0,6) circle(2pt)
(1,6) circle(2pt)
(2,6) circle(2pt)
(4,6) circle(2pt)
(5,6) circle(2pt)
(6,6) circle(2pt)

(0.5,4.5) node{$\vdots$}
(1.5,4.5) node{$\vdots$}
(4.5,4.5) node{$\vdots$}
(0.5,2) node{$\vdots$}
(1.5,2) node{$\vdots$}
(4.5,2) node{$\vdots$}

(3,3) node{$\cdots$}
(3,0) node{$\cdots$}
(3,5.5) node{$\cdots$};

\draw

(3,-1) node{(D)}
(5,0) arc(130:210:-4.7)

(2.3,0)--(0,0)--(0,1.3)
(5,3.8)--(5,3.5)--(6,3.5)--(6,2.5)--(5,2.5)--(5,2.2)

(0,2.2)--(0,3.8)
(1,2.2)--(1,3.8)
(2,2.2)--(2,3.8)
(4,2.2)--(4,3.8)

(0,4.7)--(0,6)--(1,6)--(1,4.7)
(2,4.7)--(2,6)--(2.3,6)
(1,1.3)--(1,1)--(2,1)--(2,1.3)
(3.7,1)--(4,1)--(4,1.3)
(4,4.7)--(4,6)--(3.7,6)

(3.7,0)--(5,0)
(6,6)--(5,6)--(5,5)--(6,5)--(6,4.7);

    \end{scope}

\end{tikzpicture}
\end{center}
\caption{Snake necklaces} \label{snakes}
\end{figure}


\end{proof}

Next, we explicitly construct an $(a,b)$-necklace of length $3a+b$ when $2a \leq b$. To that end, we string together residue classes $\pmod{a}$ in a particular manner; see Figure~\ref{3a+b} for an example of this pattern. 
 
 \begin{figure}[h!]
\begin{center}
\begin{tikzpicture}[scale=0.75]

\draw
(0,1) node{20}
(1,1) node{23}
(2,1) node{26}
(0,0) node{0}
(1,0) node{3}
(2,0) node{6}
(3,0) node{9}
(4,0) node{12}
(5,0) node{15}
(6,0) node{18}
(7,0) node{21}
(8,0) node{24}
(9,0) node{27}
(7,-1) node{1}
(8,-1) node{4}
(9,-1) node{7}
(10,-1) node{10}
(11,-1) node{13}
(12,-1) node{16}
(13,-1) node{19}
(14,-1) node{22}
(15,-1) node{25}
(16,-1) node{28}
(14,-2) node{2}
(15,-2) node{5}
(16,-2) node{8}
(17,-2) node{11}
(18,-2) node{14}
(19,-2) node{17};

\draw[-] (0,0.7)--(0,0.3);
\draw[-] (0.3,0)--(0.7,0);
\draw[-] (1,0.3)--(1,0.7);
\draw[-] (1.3,1)--(1.7,1);
\draw[-] (2,0.7)--(2,0.3);
\draw[-] (2.3,0)--(2.7,0);
\draw[-] (3.3,0)--(3.7,0);
\draw[-] (4.3,0)--(4.7,0);
\draw[-] (5.3,0)--(5.7,0);
\draw[-] (6.3,0)--(6.7,0);

\draw[-] (7,-0.3)--(7,-0.7);
\draw[-] (7.3,-1)--(7.7,-1);
\draw[-] (8,-0.7)--(8,-0.3);
\draw[-] (8.3,0)--(8.7,0);
\draw[-] (9,-0.3)--(9,-0.7);

\draw[-] (9.3,-1)--(9.7,-1);
\draw[-] (10.3,-1)--(10.7,-1);
\draw[-] (11.3,-1)--(11.7,-1);
\draw[-] (12.3,-1)--(12.7,-1);
\draw[-] (13.3,-1)--(13.7,-1);

\draw[-] (14,-1.3)--(14,-1.7);
\draw[-] (14.3,-2)--(14.7,-2);
\draw[-] (15,-1.7)--(15,-1.3);
\draw[-] (15.3,-1)--(15.7,-1);
\draw[-] (16,-1.3)--(16,-1.7);

\draw[-] (16.3,-2)--(16.7,-2);
\draw[-] (17.3,-2)--(17.7,-2);
\draw[-] (18.3,-2)--(18.7,-2);

\draw (19,-1.7) arc (62:100.2:29);

\end{tikzpicture}
\end{center}
\caption{Stringing together residue classes$\pmod{3}$ in a $(3,20)$-necklace of length 29.}
\label{3a+b}
\end{figure}
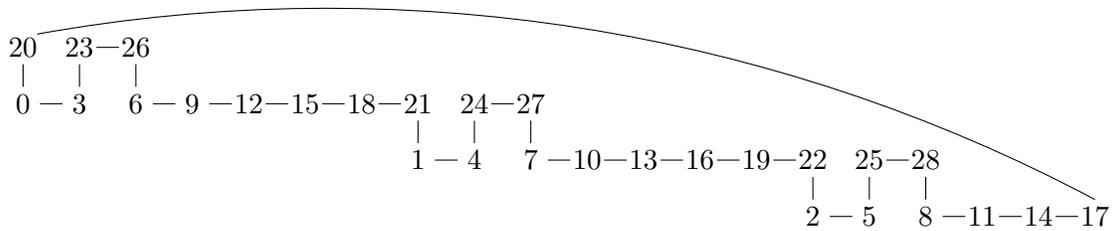

\begin{lemma} \label{L:3a+b}
Let $a,b$ be positive coprime integers with $a \ge 2$ and $2a \le b$. Then there exists an $(a,b)$-necklace of length $3a+b$. 
\end{lemma}

\begin{proof}  
For all $i \ge 0$, put $t_i \equiv -bi \pmod{a}$ with $0 \leq t_i \leq a-1$. Let $k_i$ be the unique integer such that $b \leq t_i + k_ia < b+a$. Since $t_{i+1} + b \equiv t_i \pmod{a}$ and $b \le t_{i+1} + b \le b+a-1$, we see that $t_{i+1} + b = t_i + k_ia$ for all $i \ge 0$. Moreover, $k_i a \ge b - t_i \ge 2a - (a-1) > a$, so $k_i \ge 2$.  

For each $i \geq 0$, define the sequence 
\[ S_i  : \quad t_i+b \, , \ t_i \, , \ t_i+a \, , \ t_i+a+b \, , \ t_i+2a+b \, , \quad t_i+2a \, , \ t_i+3a \, , \ \ldots \, , \ t_i+(k_i-1)a \]
of length $k_i + 3$; here, the terms starting at $t_i+2a$ are omitted when $k_i = 2$, leaving only the first five terms. Then all the terms in $S_i$ are integers between $0$ and $3a+b-1$.

We claim that the concatenation $S = S_0 S_1 \ldots S_{a-1}$ is an $(a,b)$-necklace of length $3a+b$. To see that all the terms in $S$ are distinct, note that every term in $S$ belongs to a residue class of the form $t_i \pmod{a}$ or $t_i+ b \pmod{a}$ for some $i \in \{ 0, 1, \ldots , a-1 \}$. Let $i, j$ be integers with $0 \leq i, j \leq a-1$. If $t_j \equiv t_i \pmod{a}$, then $bj \equiv bi \pmod{a}$, which implies $i = j$ as $\gcd(a,b) = 1$. Similarly, $t_j \equiv t_i+b \pmod{a}$ forces $i = j$ or $j = i+1$; in the latter case, $t_{i+1} = t_i - b + k_ia \neq t_i$. 

The length of $S$ is
\[ \sum_{i=0}^{a-1} (k_i+3) 
    = \left ( \sum_{i=0}^{a-1} \frac{t_{i+1} - t_i + b}{a} \right ) + 3a 
    = \frac{t_a - t_0}{a} + b + 3a = b+3a \ , \]
as $t_0 = t_a = 0$. Hence, $S$ is a permutation of the integers $0, 1, \ldots , 3a+b-1$. Finally, any two neighbours in $S_i$ have absolute difference $a$ or $b$. For $0 \leq i \leq a-1$, the first term in $S_{i+1}$ is $t_{i+1} + b = t_i + k_ia$ which has difference $a$ from the last term $t_i+(k_i-1)a$ in $S_i$; since $S_a = S_0$. This proves our claim.

\end{proof} 

Note that the construction of the $(a,b)$-necklace of length $3a+b$ in Lemma~\ref{L:3a+b} relies crucially on the fact that $k_i \geq 2$ for all $i$, which was guaranteed by the inequality $2a \le b$. The approach fails if $k_i = 1$ for some $i$, since $t_{i+1} + b = t_i + a$ in this case, so $S_i$ and $S_{i+1}$ are no longer disjoint. For example, it is easy to check that there is no $(2,3)$ necklace of length~9. 

If $a \geq 2$, then two $(a,b)$-necklaces can be combined to create a longer $(a,b)$-necklace by adding the length of one necklace to all the beads in the other and then gluing along a pair of suitable links. Specifically, two $(a,b)$-necklaces $X$ and $Y$ of respective lengths $m$ and $n$ can produce an $(a,b)$-necklace of length $m+n$ if $X$ contains two adjacent beads $x, x'$ whose difference is $\pm a$ such that $n-b+x,n-b+x'$ are adjacent beads in $Y$. In this case, the beads in $X$ are shifted by~$n$ and then the two necklaces are glued together at the four specified beads, with $n+x$ placed adjacent to $n-b+x$ and $n+x'$ next to $n-b+x'$ in the new necklace of length~$m+n$. Figure~\ref{glued} shows two $(2,7)$-necklaces of lengths 9 and 13, respectively. Beads $5, 3$ are adjacent in the length 9 necklace (on the left) and beads $11, 9$ are adjacent in the length 13 necklace (on the right). In Figure~\ref{glued2}, all the beads in the length 9 necklace are shifted up by 13; in particular, 5, 3 are shifted to 18, 16 and attached to 11, 9, respectively, to form a $(2,7)$-necklace of length~22.

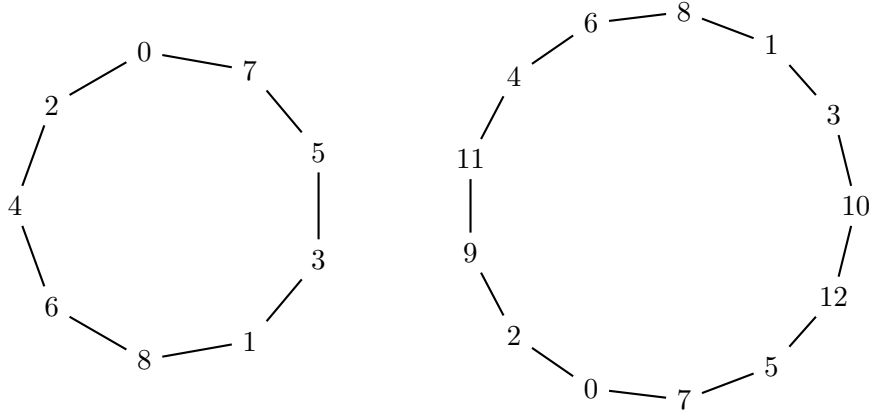
\begin{figure}[h!]
\begin{center}
\begin{tikzpicture}[-,thick,scale=1.3]
\node (10) at (5*360/9 + 360/18: 1.6cm) {6};
\node (11) at (4*360/9 + 360/18: 1.6cm) {4};
\node (12) at (3*360/9 + 360/18: 1.6cm) {2};
\node (13) at (2*360/9 + 360/18: 1.6cm) {0};
\node (14) at (1*360/9 + 360/18: 1.6cm) {7};
\node (15) at (0*360/9 + 360/18: 1.6cm) {5};
\node (16) at (-1*360/9 + 360/18: 1.6cm) {3};
\node (17) at (-2*360/9 + 360/18: 1.6cm) {1};
\node (18) at (-3*360/9 + 360/18: 1.6cm) {8};

\draw[-] (11) to (10);
\draw[-] (12) to (11);
\draw[-] (13) to (12);
\draw[-] (14) to (13);
\draw[-] (15) to (14);
\draw[-] (16) to (15);
\draw[-] (17) to (16);
\draw[-] (18) to (17);
\draw[-] (10) to (18);

    \begin{scope}[xshift=50mm]
\node (20) at (5*360/13 : 2cm) {4};
\node (21) at (4*360/13 : 2cm) {6};
\node (22) at (3*360/13 : 2cm) {8};
\node (23) at (2*360/13 : 2cm) {1};
\node (24) at (1*360/13: 2cm) {3};
\node (25) at (0*360/13 : 2cm) {10};
\node (26) at (-1*360/13 : 2cm) {12};
\node (27) at (-2*360/13 : 2cm) {5};
\node (28) at (-3*360/13 : 2cm) {7};
\node (29) at (-4*360/13 : 2cm) {0};
\node (30) at (-5*360/13 : 2cm) {2};
\node (31) at (-6*360/13 : 2cm) {9};
\node (32) at (-7*360/13 : 2cm) {11};

\draw[-] (21) to (20);
\draw[-] (22) to (21);
\draw[-] (23) to (22);
\draw[-] (24) to (23);
\draw[-] (25) to (24);
\draw[-] (26) to (25);
\draw[-] (27) to (26);
\draw[-] (28) to (27);
\draw[-] (29) to (28);
\draw[-] (30) to (29);
\draw[-] (31) to (30);
\draw[-] (32) to (31);
\draw[-] (20) to (32);
    \end{scope}


\end{tikzpicture}
\end{center} 
\caption{Two $(2,7)$-necklaces of respective lengths 9 and 13.}
\label{glued}
\end{figure}

\begin{figure}[h!]
\begin{center}
\begin{tikzpicture}[-,thick,scale=1.3]
\node (10) at (5*360/9 + 360/18: 1.6cm) {19};
\node (11) at (4*360/9 + 360/18: 1.6cm) {17};
\node (12) at (3*360/9 + 360/18: 1.6cm) {15};
\node (13) at (2*360/9 + 360/18: 1.6cm) {13};
\node (14) at (1*360/9 + 360/18: 1.6cm) {20};
\node (15) at (0*360/9 + 360/18: 1.6cm) {18};
\node (16) at (-1*360/9 + 360/18: 1.6cm) {16};
\node (17) at (-2*360/9 + 360/18: 1.6cm) {14};
\node (18) at (-3*360/9 + 360/18: 1.6cm) {21};

\draw[-] (11) to (10);
\draw[-] (12) to (11);
\draw[-] (13) to (12);
\draw[-] (14) to (13);
\draw[-] (15) to (14);
\draw[-] (17) to (16);
\draw[-] (18) to (17);
\draw[-] (10) to (18);

    \begin{scope}[xshift=50mm]
\node (20) at (5*360/13 : 2cm) {4};
\node (21) at (4*360/13 : 2cm) {6};
\node (22) at (3*360/13 : 2cm) {8};
\node (23) at (2*360/13 : 2cm) {1};
\node (24) at (1*360/13: 2cm) {3};
\node (25) at (0*360/13 : 2cm) {10};
\node (26) at (-1*360/13 : 2cm) {12};
\node (27) at (-2*360/13 : 2cm) {5};
\node (28) at (-3*360/13 : 2cm) {7};
\node (29) at (-4*360/13 : 2cm) {0};
\node (30) at (-5*360/13 : 2cm) {2};
\node (31) at (-6*360/13 : 2cm) {9};
\node (32) at (-7*360/13 : 2cm) {11};

\draw[-] (21) to (20);
\draw[-] (22) to (21);
\draw[-] (23) to (22);
\draw[-] (24) to (23);
\draw[-] (25) to (24);
\draw[-] (26) to (25);
\draw[-] (27) to (26);
\draw[-] (28) to (27);
\draw[-] (29) to (28);
\draw[-] (30) to (29);
\draw[-] (31) to (30);
\draw[-] (20) to (32);
    \end{scope}

\draw[-] (16) to (31) ;
\draw[-] (32) to (15) ;

\end{tikzpicture}
\end{center} 
\caption{The two $(2,7)$-necklaces of Figure \ref{glued} glued together to form a $(2,7)$-necklace of length 22, with all the beads in the length 9 necklace shifted up by 13.}
\label{glued2}
\end{figure}
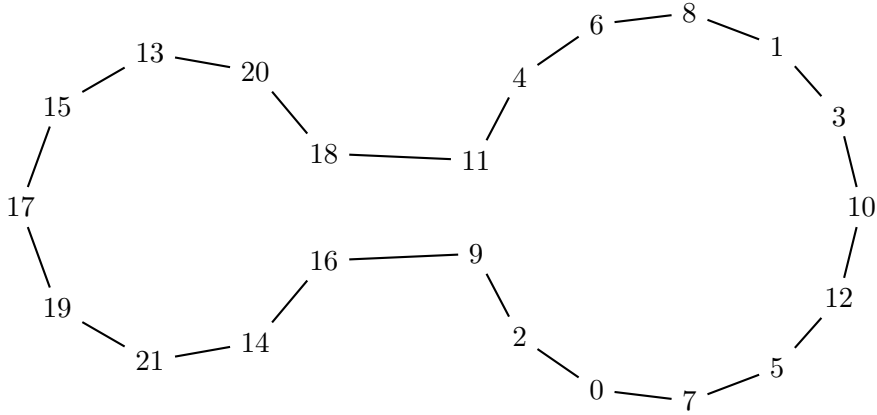

\begin{lemma} \label{L:glue} 
Let $X$ and $Y$ be two $(a,b)$-necklaces of respective lengths $m$ and $n$. If $a \ge 2$, then there exists an $(a,b)$-necklace of length $m+n$. 
\end{lemma} 
\begin{proof} 
Write $b = qa + r$ where $q,r$ are positive integers with $r \le a-1$. Put
\[ x_i = (q-i)a  , \quad y_i = n - b + x_i \qquad (0 \leq i \leq q) \ , \]
and consider the $x_i$ as beads in $X$ and the $y_i$ as beads in $Y$. Since $x_i < b$ and $y_i \ge n-b$, the neighbours of $x_i$ in $X$ are two among $x_i+b$, $x_i+a$, $x_i-a = x_{i+1}$ and the neighbours of $y_i$ in $Y$ are two among $y_i-b$, $y_i+a$, $y_i-a = y_{i+1}$. To prove the lemma, it suffices to establish the existence of an index $i$ with $0 \leq i \leq q-1$ such that $x_i, x_{i+1}$ are adjacent in~$X$ and $y_i, y_{i+1}$ are adjacent in $Y$, for $0 \leq i \le q-1$.

In $Y$, $y_0 = n-b+qa = n-r$ has neighbours $y_0-b$ and $y_0-a = y_1$. If $x_0$ and $x_1$ are adjacent in $X$, then the claim is proved, so suppose that this is not the case. Define the set of integers
\[ I = \{ i \mid 0 \le i \le \left \lfloor \frac{q-1}{2} \right \rfloor, \mbox{ $x_{2i}$ is not adjacent to $x_{2i+1}$ in $X$, $y_{2i}$ is adjacent to $y_{2i+1}$ in $Y$} \} \, . \]
Then $0 \in I$. Let $k = \max \{ i \mid i \in I \} \ge 0$ be the largest element of $I$. Then $y_{2k}$,  $y_{2k+1}$ are adjacent in $Y$ and $x_{2k}, x_{2k+1}$ are not adjacent in $X$. Thus, $x_{2k+1}$ has neighbours $x_{2k+1}+b$ and $x_{2k+1}+a = x_{2k+2}$. If $y_{2k+1}$ and $y_{2k+2}$ are neighbours in $Y$, then the lemma is again proved, so suppose they are non-adjacent beads in $Y$. Then $y_{2k+2}$ has neighbours $y_{2k+2}-b$ and $y_{2k+2}-a =  y_{2k+3}$. Since $k+1 \notin I$, this forces $x_{2k+2}$ and $x_{2k+3}$ to be adjacent in $X$, which once again proves the lemma. \end{proof}

Note that this gluing technique does not work in all cases. For example, the $(1,5)$-necklace $0,1,6,7,2,3,8,9,4,5$ of length~10 cannot be glued to itself in this manner, since it contains no adjacent beads $x,x'$ such that $5+x, 5+x'$ are also adjacent. The requirement that $a \ge 2$ in Lemma \ref{L:glue} is a sufficient but not a necessary condition. For example, the $(1,5)$-necklace $0,1,2,3,4,5$ can be glued to itself. In fact, any $(a.b)$-necklace of length $a+b$ can be glued to itself; see Theorem~\ref{expest}. 

We now have all the ingredients for our main existence result for $(a,b)$-necklaces.

\begin{theorem}\label{exist} 
Let $a$ and $b$ be positive coprime integers with $2a \leq b$. If $ab$ is even, then there exists an $(a,b)$-necklace of length $n$ for all $n \ge (a+b-1)(3a+b-1)$. If $ab$ is odd, then there exists an $(a,b)$-necklace of even length $n$ for all $n \ge (a+b-2)(3a+b-2)/2$.
\end{theorem}
\begin{proof} Lemma~\ref{a=1} states this result for $a=1$, so assume now that $a \geq 2$. Then there exists an $(a,b)$-necklace of length $a+b$ and one of length $3a+b$ by Lemma~\ref{L:3a+b}. By Lemma~\ref{L:glue}, there exists an $(a,b)$-necklace of length $x(a+b) + y(3a+b)$ for all $x,y \geq 0$.

If $ab$ is even, then $a+b$ and $3a+b$ are coprime, so by the solution to the well-known Frobenius Coin Problem~\cite[Theorem 2.1.1]{JR}, there exists an $(a,b)$-necklace of length $n$ for all $n \geq (a+b-1)(3a+b-1)$. If $ab$ is odd, then gcd$(a+b,3a+b) = 2$, so  there exists an $(a,b)$-necklace of even length $2m$ for all $m \geq \big ( (a+b)/2-1 \big ) \big ( (3a+b)/2-1 \big)$.
\end{proof}

The lower bounds in Theorem~\ref{exist} are far from tight. For example, the theorem asserts the existence of a $(2,3)$-necklace of any length $n \geq 32$, but in the next section, we will see that $(2,3)$-necklaces of any length $n \geq 10$ exist. 

We strongly believe that the restriction $2a \le b$ is not necessary to guarantee the existence of necklaces of any sufficiently large length. Numerical computations seem to support this assertion.

\begin{conjecture}\label{conje} 
Let $a$ and $b$ be positive coprime integers. Then there is a positive integer~$n_{a,b}$ such that an $(a,b)$-necklace of every length $n \geq n_{a,b}$ exists when $ab$ is even and an $(a,b)$-necklace of every even length $n \geq n_{a,b}$ exists when $ab$ is odd.
\end{conjecture}

Recall that $N_{a,b}(n)$ denotes the number of distinct $(a,b)$-necklaces of length $n$, or equivalently, the number of Hamiltonian cycles in $G_{a,b}(n)$. We can apply our gluing idea above to necklaces whose length is a multiple of $a+b$ to obtain a lower bound on $N_{a,b}(n)$ in this case. 

\begin{theorem}\label{expest} Let $a$ and $b$ be positive coprime integers. Then for all positive integers $k$ we have
\[ N_{a,b}(k(a+b)) \geq \begin{cases} (b-a)^{k-1} & \text{if $2a \geq b$ or $k=1$} , \\ (b-a)(b-a-1)^{k-2} & \text{if $2a<b$ and $k \geq 2$} . \end{cases}\]
\end{theorem}

\begin{proof} Let $X$ be the unique $(a,b)$-necklace of length $a+b$. For any $m \in \mathbb{N}$, let $X+m$ denote the necklace obtained from $X$ by adding $m$ to each bead. Our aim is to determine the number of distinct ways in which $X + (k-1)(a+b)$ can be inserted into a necklace of length $(k-1)(a+b)$ via gluing to yield distinct necklaces of length $k(a+b)$. The result will then follow inductively. 

We first consider the case $k = 2$, where $X$ is glued to itself. Inserting $X+a+b$ into $X$ requires a link $x+a,x+2a$ in $X$ and a link $x+a+b,x+2a+b$ in $X+a+b$; the gluing process then attaches $x+a$ to $x+a+b$ and $x+2a$ to $x+2a+b$; see Figure~\ref{glued3}. The existence of such a pair of links is equivalent to the subsequence $x,x+a,x+2a$ appearing in $X$. Thus, the number of ways to insert $X+a+b$ into $X$ is bounded below by the number of subsequences of the form $x,x+a,x+2a$ appearing in $X$. It is easy to see that $X$ has the following adjacencies:

\begin{itemize} \itemsep -3pt
\item[] Beads $x$ with $0 \le x < a$ have neighbours $x+a$ and $x+b$;
\item[] Beads $x$ with $a \le x < b$ have neighbours $x+a$ and $x-a$;
\item[] Beads $x$ with $b \le x < a+b$ have neighbours $x-a$ and $x-b$.
\end{itemize}
Therefore, for any $x$ with $0 \le x < b-a$, the subsequence $x,x+a,x+2a$ appears in $X$. It follows that $X+a+b$ can be inserted into $X$ in at least $b-a$ different ways, yielding at least $b-a$ distinct $(a,b)$-necklaces of length $b-a$. 

For the case $k = 3$, let $Y$ be an $(a,b)$-necklace of length $2(a+b)$ obtained via the above construction, and let $x+a,x+a+b$ and $x+2a,x+2a+b$ be the links in $Y$ created by inserting $X + a+b$ into $X$.
To insert $X+2(a+b)$ into $Y$, there must exist an integer $y \in [a+b,2b-1]$ such that $y+a,y+2a$ is a link in $Y$ and $y+a+b,y+2a+b$ is a link in $X+2(a+b)$. This requires $y-a-b,y-b,y+a-b$ to be a subsequence in $X$. By our reasoning above, there are $b-a$ such subsequences in $X$. However, we must exclude the case $(y+a,y+2a) = (x+a+b,x+2a+b)$ since this is the link that was broken when inserting $X+a+b$ into $X$. If $(y+a,y+2a) = (x+a+b,x+2a+b)$, then the sequence $x,x+a,x+2a,x+3a$ appears in $X$, forcing $3a<a+b$. Therefore, if $2a \geq b$, then $X+2(a+b)$ can be inserted into $Y$ in at least $b-a$ different ways, whereas if $2a<b$, there are $b-a-1$ distinct possibilities for such an insertion. It follows that there are at least $(b-a)^2$ necklaces of length $3(a+b)$ when $2a \geq b$ and at least $(b-a)(b-a-1)$ such necklaces when $2a \geq b$. The theorem now follows by inductively inserting shifts of $X$ into itself. 
\end{proof}

\begin{figure}[h!]
    \centering
    \begin{subfigure}[t]{0.6\textwidth}
        \centering
     \begin{tikzpicture}[scale=0.6]
\draw (0,1.5) node {$x$};
\draw (0,0) node {$x$+$a$};
\draw (3,1.5) node {$x$+$a$};
\draw (3,0) node {$x$+$2a$};

\draw[-]  (-2,1.5)--(-1,1.5) (1,1.5)--(2,1.5) (4,1.5)--(5,1.5); 
\draw[-]  (-2,0)--(-1,0) (1,0)--(2,0) (4,0)--(5,0);
    \end{tikzpicture}
        \caption{Two fragments of the $(a,b)$-necklace of length $a+b$.}
    \end{subfigure}%
    ~ 
    \begin{subfigure}[t]{0.4\textwidth}
        \centering
     \begin{tikzpicture}[scale=0.6]
\draw (0.5,1.5) node {$x$+$a$+$b$};
\draw (0.5,0) node {$x$+$a$};
\draw (4,1.5) node {$x$+$2a$+$b$};
\draw (4,0) node {$x$+$2a$};

\draw[-]  (-2,1.5)--(-1,1.5) (0.5,1.1)--(.5,0.4) (5.5,1.5)--(6.5,1.5); 
\draw[-]  (-2,0)--(-1,0) (4,1.1)--(4,0.4) (5.5,0)--(6.5,0);

    \end{tikzpicture}
        \caption{Gluing along the two fragments.}
    \end{subfigure}
    \caption{}
    \label{glued3}
\end{figure}

Tables~\ref{table1} and~\ref{table2} suggest that the bound in Theorem~\ref{expest} is sharp for $k = 2$ except when $a = 1$ and $b$ is odd, where we always obtained $N_{1,b}(2(b+1)) = b$. This is because an additional $(1,b)$-necklace arises from snake pattern (D) in Figure~\ref{snakes}. In this necklace, bead $b+1$ is adjacent to $b$ and $2b+1$, whereas in every necklace obtained from the gluing construction, $b+1$ is adjacent to $b+2$ and $2b+1$. 

Theorem~\ref{expest} shows that $N_{a,b}(n)$ is unbounded for all permissible pairs $a,b$ unless $(a,b) = (1,3)$ or $a+1 = b$. Due to the resemblance of $G_{a,a+1}(n)$ to a grid graph, the same result is also almost certainly true when $b = a+1$ except for the case $(a,b) = (1,2)$.

In the following section we will determine $N_{a,b}(n)$ explicitly for some pairs~$(a,b)$.


\section{Counting $(a,b)$-Necklaces} \label{S:count}

In this section, we derive linear recurrence relations for $N_{a,b}(n)$ when $(a,b) \in \{(1,2)$, $(1,3)$,$(2,3)$,$(1,4)\}$. We assume $n \ge a+b$ throughout, as $N_{a,b}(n) = 0$ for $n < a+b$.

Our counting technique makes use of the following construction. Let $k$ be a positive integer and $H$ a path or a cycle whose vertices are non-negative integers. For brevity, let $H+k$ denote the path or cycle obtained by adding $k$ to each vertex in $H$; $H-k$ is similarly defined (and may have negative vertices). Now define $S_k(H)$ to be the path or cycle obtained from $H-k$ by removing all the negative vertices in $H-k$. For example, for $k = 4$ and the path $H$ given by $7-3-4-0-2-5-9-6$, the path $S_4(H)$ is $3-0-1-5-2$. Counting Hamiltonian cycles in $G_{a,b}(n)$ is greatly facilitated by the following simple but crucial observation.

\begin{lemma} \label{downshift}
Let $a, b, k, n$ be positive integers with $a < b$, $\gcd(a,b) = 1$ and $k \ge n-3$. Let~$P$ be a path in $G_{a,b}(n)$ such that every element of $\{ 0, 1, \ldots , k-1 \}$ is an internal vertex of~$P$ and~$S_k(P)$ is a sub-graph of~$G_{a,b}(n-k)$. Then the Hamiltonian cycles  in~$G_{a,b}(n)$ that contain~$P$ are in one-to-one correspondence with the Hamiltonian cycles in~$G_{a,b}(n-k)$ that contain~$S_k(P)$ via the correspondence $C \rightarrow S_k(C)$.
\end{lemma}
\begin{proof}
The vertices of $S_k(P)$ are integers between 0 and $n-k-1$. When constructing~$S_k(P)$ from~$P-k$, the downshifts of the $k$ internal vertices $0, 1, \ldots , k-1$, and only those, are removed from $P$. In particular, the downshifts by $k$ of the end points of $P$ are the end points of $S_k(P)$. 

Let $C$ be a Hamiltonian cycle in $G_{a,b}(n)$ containing $P$. Then $S_k(C)$ is a cycle with vertex set $\{ 0, 1, \ldots , n-k-1 \}$ that contains $S_k(P)$ and is easily verified to be a subgraph, and hence a Hamiltonian cycle, of $G_{a,b}(n-k)$. Conversely, let $C$ be a Hamiltonian cycle in $G_{a,b}(n-k)$ containing $S_k(P)$. Then $C+k$ is a cycle with vertex set $\{k, k+1, \ldots , n-1\}$ that contains $S_k(P)+k$. Replacing $S_k(P)+k$ by $P$ in $C+k$ adds the vertices $0, 1, \ldots , k-1$ and is readily seen to produce a Hamiltonian cycle in $G_{a,b}(n)$.
\end{proof}

For example, consider the path $P$ in $G_{1,4}(n)$ given by $9-5-6-2-1-0-4-3-7$. Here, 0, 1, 2, 3, 4 are internal vertices of $P$, so we can construct the path $S_5(P)$ given by $4-0-1-2$ which is also contained in $G_{1,4}(n)$. By Lemma \ref{downshift}, the Hamiltonian cycles in $G_{1,4}(n)$ containing $9-5-6-2-1-0-4-3-7$ are in one-to-one correspondence with the Hamiltonian cycles in $G_{1,4}(n-5)$ containing $4-0-1-2$. In general, note that $S_k(P)$ is not always a subgraph of $G_{a,b}(n-k)$. For example, consider the path $P$ in $G_{1,3}(n)$ given by $4-3-0-1-2-5$. Now $S_3(P)$ is $1-0-2$, which is not a path in $G_{1,3}(n-3)$.

\subsection{Counting $(1,2)$-Necklaces}

The case of $(1,2)$-necklaces is straightforward. Every Hamiltonian cycle in $G_{1,2}(n)$ contains the path $1-0-2$, and it is now easy to see that this determines the cycle uniquely, traced around the periphery of $G_{1,2}(n)$. Hence,
\[ N_{1,2}(n) = 1 \quad \mbox{for all $n \geq 3$} . \]

\subsection{Counting $(1,3)$-Necklaces}

Only Hamiltonian cycles of even length exist in this setting. Every Hamiltonian cycle $C$ in $G_{1,3}(n)$ contains the path $1-0-3$. Vertex 1 is thus adjacent to either 2 or 4 in $C$, so $C$ contains exactly one of the paths $2-1-0-3$ or $4-1-0-3$. If $n \ge 6$, then 2 is easily seen to be adjacent to 5 in $C$. Hence, if $C$ contains $2-1-0-3$, then it contains $5-2-1-0-3$. If $C$ contains $4-1-0-3$, then 2 has neighbours 3 and 5, and if $n \geq 8$, then 4 is easily verified to be adjacent to 7; hence $C$ contains $7-4-1-0-3-2-5$. 

If $C$ contains $5-2-1-0-3$, then $C-2$ contains $3-0-1$ and thus corresponds to a unique Hamiltonian cycle in $G_{1,3}(n-2)$ by Lemma \ref{downshift}. Similarly, if $C$ contains $7-4-1-0-3-2-5$, then $C-4$ contains $3-0-1$ and hence again corresponds to a unique Hamiltonian cycle in $G_{1,3}(n-4)$ by Lemma \ref{downshift}. Let $N_{1,3}^{(2)}(n)$ and $N_{1,3}^{(4)}(n)$ be the number of Hamiltonian cycles in $G_{1,3}(n)$ that contain $5-2-1-0-3$ and $7-4-1-0-3-2-5$, respectively. These quantities and their relationship to each other is illustrated in Figure~\ref{13cases}.

\begin{figure}[h!]
\begin{center}
\begin{tikzpicture}[scale=1]

\node[draw][circle](1) at (0,4) {0};
\node[draw][circle](4) at (1,4) {3};
\node[draw][circle](7) at (2,4) {6};

\node[draw][circle](2) at (0,3) {1};
\node[draw][circle](5) at (1,3) {4};
\node[draw][circle](8) at (2,3) {7};

\node[draw][circle](3) at (0,2) {2};
\node[draw][circle](6) at (1,2) {5};
\node[draw][circle](9) at (2,2) {8};

\node (10) at (3,4) {} ;
\node (11) at (3,3) {};
\node (12) at (3,2) {};
\node at (1,1) {$N_{1,3}(n)$};


\draw[dotted]
(1)--(2)--(3)
(4)--(5)--(6)
(7)--(8)--(9)
(1)--(4)--(7)--(10)
(2)--(5)--(8)--(11)
(3)--(6)--(9)--(12)
(3)--(4)
(6)--(7)
(9)--(10);

\draw[very thick]
(4)--(1)--(2);


\def\x{5}

\node[draw][circle](13) at (\x+0,4) {0};
\node[draw][circle](16) at (\x+1,4) {3};
\node[draw][circle](19) at (\x+2,4) {6};

\node[draw][circle](14) at (\x+0,3) {1};
\node[draw][circle](17) at (\x+1,3) {4};
\node[draw][circle](20) at (\x+2,3) {7};

\node[draw][circle](15) at (\x+0,2) {2};
\node[draw][circle](18) at (\x+1,2) {5};
\node[draw][circle](21) at (\x+2,2) {8};

\node (22) at (\x+3,4) {} ;
\node (23) at (\x+3,3) {};
\node (24) at (\x+3,2) {};
\node at (\x+1,1) {$N_{1,3}^{(2)}(n)$};
\node at (\x*0.5+1,1) {$=$};


\draw[dotted]
(13)--(14)--(15)
(16)--(17)--(18)
(19)--(20)--(21)
(13)--(16)--(19)--(22)
(14)--(17)--(20)--(23)
(15)--(18)--(21)--(24)
(15)--(16)
(18)--(19)
(21)--(22);

\draw[very thick]
(16)--(13)--(14)--(15)--(18);


\def\y{9}

\node[draw][circle](13) at (\y+0,4) {0};
\node[draw][circle](16) at (\y+1,4) {3};
\node[draw][circle](19) at (\y+2,4) {6};

\node[draw][circle](14) at (\y+0,3) {1};
\node[draw][circle](17) at (\y+1,3) {4};
\node[draw][circle](20) at (\y+2,3) {7};

\node[draw][circle](15) at (\y+0,2) {2};
\node[draw][circle](18) at (\y+1,2) {5};
\node[draw][circle](21) at (\y+2,2) {8};

\node (22) at (\y+3,4) {} ;
\node (23) at (\y+3,3) {};
\node (24) at (\y+3,2) {};
\node at (\y+1,1) {$N_{1,3}^{(3)}(n)$};
\node at (\x*0.5+\y*0.5+1,1) {$+$};


\draw[dotted]
(13)--(14)--(15)
(16)--(17)--(18)
(19)--(20)--(21)
(13)--(16)--(19)--(22)
(14)--(17)--(20)--(23)
(15)--(18)--(21)--(24)
(15)--(16)
(18)--(19)
(21)--(22);

\draw[very thick]
(18)--(15)--(16)--(13)--(14)--(17)--(20);

\end{tikzpicture}
\end{center}
\caption{Counting $(1,3)$-Necklaces.}
\label{13cases}
\end{figure}
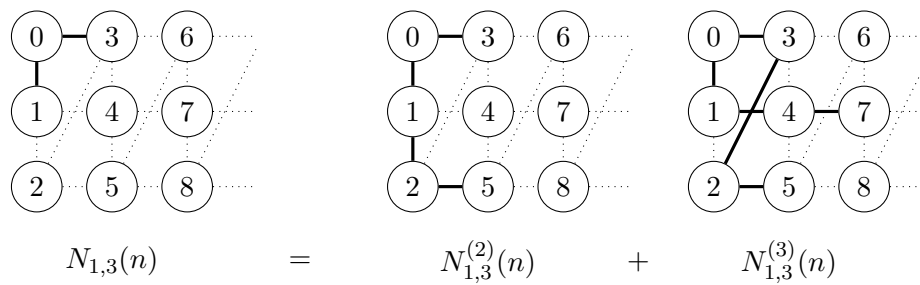

We thus see that
\begin{equation*}
    N_{1,3}(n) = N_{1,3}(n-2) + N_{1,3}(n-4) \qquad (n \ge 8) \ .
\end{equation*}
It is easy to check that $N_{1,3}(4)=1$ and $N_{1,3}(6) = 2$, so 
\[ N_{1,3}(n) = \begin{cases} F_{n/2} & \text{if } n \ge 4 \text{ is even,}\\ 0 & \text{if } n \ge 5 \text{ is odd,}
 \end{cases} 
 \]
where $F_n$ denotes the $n^{\textrm{th}}$ Fibonacci number (with $F_0 = 0$ and $F_1 = 1$).


\subsection{Counting $(2,3)$-Necklaces}

By considering the neighbours of 0 and 1 in any Hamiltonian cycle $C$ of $G_{2,3}(n)$, we see that~$C$ contains the path $4-1-3-0-2$, which extends to $4-1-3-0-2-5$ when $n \ge 6$. In addition, 4 is adjacent to either 6 or 7 in $C$. In the latter case, 5 is adjacent to 8 to avoid the cycle $7-4-1-3-0-2-5$. Moreover, 6 must be adjacent to 8 and 9, which forces 7 to be adjacent to 10 and shows that $C$ contains the path $10-7-4-1-3-0-2-5-8-6-9$. Similar to the previous case, we split the collection of Hamiltonian cycles in $G_{2,3}(n)$ into two subsets, depending on whether 4 has neighbour 6 or 7 in the cycle. Let $N_{2,3}^{(2)}(n)$ and $N_{2,3}^{(3)}(n)$ count the number of Hamiltonian cycles containing the respective paths $6-4-1-3-0-2-5$ and $10-7-4-1-3-0-2-5-8-6-9$. Then we obtain the situation of Figure~\ref{23cases}.

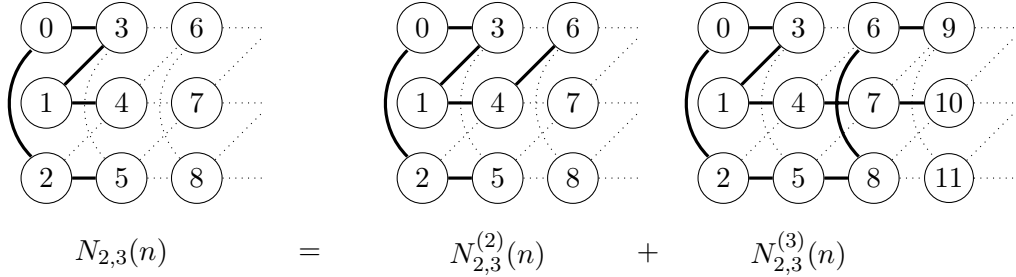
\begin{figure}[h!]
\begin{center}
\begin{tikzpicture}[scale=1]

\node[draw][circle](1) at (0,4) {0};
\node[draw][circle](4) at (1,4) {3};
\node[draw][circle](7) at (2,4) {6};

\node[draw][circle](2) at (0,3) {1};
\node[draw][circle](5) at (1,3) {4};
\node[draw][circle](8) at (2,3) {7};

\node[draw][circle](3) at (0,2) {2};
\node[draw][circle](6) at (1,2) {5};
\node[draw][circle](9) at (2,2) {8};

\node (10) at (3,4) {} ;
\node (11) at (3,3) {};
\node (12) at (3,2) {};
\node at (1,1) {$N_{2,3}(n)$};


\draw[dotted]
(2)--(4)
(3)--(5)
(5)--(7)
(6)--(8)
(8)--(10)
(9)--(11)
(1)--(4)--(7)--(10)
(2)--(5)--(8)--(11)
(3)--(6)--(9)--(12)

(0.8,3.65) arc (135:225:1)
(1.8,3.65) arc (135:225:1);

\draw[very thick]
(5)--(2)--(4)--(1)
(-0.2,3.7) arc (135:225:1)
(3)--(6);


\def\x{5}

\node[draw][circle](13) at (\x+0,4) {0};
\node[draw][circle](16) at (\x+1,4) {3};
\node[draw][circle](19) at (\x+2,4) {6};

\node[draw][circle](14) at (\x+0,3) {1};
\node[draw][circle](17) at (\x+1,3) {4};
\node[draw][circle](20) at (\x+2,3) {7};

\node[draw][circle](15) at (\x+0,2) {2};
\node[draw][circle](18) at (\x+1,2) {5};
\node[draw][circle](21) at (\x+2,2) {8};

\node (22) at (\x+3,4) {} ;
\node (23) at (\x+3,3) {};
\node (24) at (\x+3,2) {};
\node at (\x+1,1) {$N_{2,3}^{(2)}(n)$};
\node at (\x*0.5+1,1) {$=$};


\draw[dotted]
(14)--(16)
(15)--(17)
(17)--(19)
(18)--(20)
(20)--(22)
(21)--(23)
(13)--(16)--(19)--(22)
(14)--(17)--(20)--(23)
(15)--(18)--(21)--(24)

(0.8+\x,3.7) arc (135:225:1)
(1.8+\x,3.7) arc (135:225:1);

\draw[very thick]
(17)--(14)--(16)--(13)
(-0.2+\x,3.7) arc (135:225:1)
(17)--(19)

(15)--(18);


\def\y{9}

\node[draw][circle](13) at (\y+0,4) {0};
\node[draw][circle](16) at (\y+1,4) {3};
\node[draw][circle](19) at (\y+2,4) {6};

\node[draw][circle](14) at (\y+0,3) {1};
\node[draw][circle](17) at (\y+1,3) {4};
\node[draw][circle](20) at (\y+2,3) {7};

\node[draw][circle](15) at (\y+0,2) {2};
\node[draw][circle](18) at (\y+1,2) {5};
\node[draw][circle](21) at (\y+2,2) {8};

\node[draw][circle] (22) at (\y+3,4) {9} ;
\node[draw][circle,inner sep=2pt] (23) at (\y+3,3) {10};
\node[draw][circle,inner sep=2pt] (24) at (\y+3,2) {11};

\node (25) at (\y+4,4) {} ;
\node (26) at (\y+4,3) {};
\node (27) at (\y+4,2) {};

\node at (\y+1,1) {$N_{2,3}^{(3)}(n)$};
\node at (\x*0.5+\y*0.5+1,1) {$+$};


\draw[dotted]
(14)--(16)
(15)--(17)
(17)--(19)
(18)--(20)
(20)--(22)
(21)--(23)
(23)--(25)
(24)--(26)
(13)--(16)--(19)--(22)--(25)
(14)--(17)--(20)--(23)--(26)
(15)--(18)--(21)--(24)--(27)

(0.8+\y,3.7) arc (135:225:1)

(2.8+\y,3.7) arc (135:225:1);

\draw[very thick]
(17)--(14)--(16)--(13)
(17)--(20)
(18)--(21)
(-0.2+\y,3.7) arc (135:225:1)
(1.8+\y,3.7) arc (135:225:1)
(15)--(18)
(19)--(22)
(20)--(23);

\end{tikzpicture}
\end{center}
\caption{Counting $(2,3)$-necklaces.}
\label{23cases}
\end{figure}

By Lemma~\ref{downshift}, we have $N_{2,3}^{(3)}(n) = N_{2,3}(n-5)$. To determine $N_{2,3}^{(2)}(n)$, consider the path $P$ given by  $6-4-1-3-0-2-5$. Here, $S_1(P)$ is the path $5-3-0-2-1-4$ which is not contained in $G_{2,3}(m)$ for any $m$. However, the permutation exchanging~2 and~3 in $S_1(P)$ produces the path $5-2-0-3-1-4$ which is contained in $G_{2,3}(n)$. So it is evident that via the bijection $C \rightarrow S_5(C)$ followed by the permutation that swaps 2 and 3 in $S_5(C)$, the Hamiltonian cycles in $G_{2,3}(n-1)$ containing $6-4-1-3-0-2-5$ are in one-to-one correspondence with the Hamiltonian cycles in $G_{2,3}(n-1)$ that contain $5-2-0-3-1-4$. It follows that $N_{2,3}^{(2)}(n) = N_{2,3}(n-1)$, and hence
\begin{equation*}
    N_{2,3}(n) = N_{2,3}(n-1) + N_{2,3}(n-5)  \qquad (n \ge 8) \ .
\end{equation*}
It is easy to find initial values of $N_{2,3}(n)$ by hand, they are shown in Table~\ref{23terms}. We remark that the sequence given by $N_{2,3}(n)$ is A017899 in OEIS. In addition A003520 is the same modulo a shift. Note that the denominator of the generating function of $N_{2,3}(n)$ is irreducible and $N_{2,3}(n)$ cannot satisfy a linear recurrence of order less than 5.

\begin{table} \centering
 \begin{tabular}{| c || c | c |c |} 
 \hline
         $n$  & $5$ & $6$ & $7$ \\ 
 \hline
 $N_{2,3}(n)$ & $1$ & $0$ & $0$  \\ 
 \hline
\end{tabular}
\caption{Initial values of $N_{2,3}(n)$.}
\label{23terms}
\end{table}

\subsection{Counting $(1,4)$-Necklaces} \label{count14}
 
Similar to the case of $(1,3)$-necklaces, we partition the set of all Hamiltonian cycles in~$G_{1,4}(n)$ into two types, depending on the neighbours of vertex~1. Every Hamiltonian cycle $C$ in $G_{1,4}(n)$ contains the path $4-0-1$ and either the edge $1-2$ or the edge $1-5$. In the former case, $C$ contain the path $4-0-1-2$; in the latter case, $C$ contains the paths $4-0-1-5$ and $3-2-6$. These two cases are depicted in Figure~\ref{14cases1}, with $N_{1,4}^{(2)}(n)$ and $N_{1,4}^{(2)}(n)$ denoting the respective counts of Hamiltonian cycles in $G_{1,4}(n)$.

\begin{figure}[h!]
\begin{center}
\begin{tikzpicture}[scale=1]

\node[draw][circle](1) at (0,4) {0};
\node[draw][circle](5) at (1,4) {4};
\node[draw][circle](9) at (2,4) {8};

\node[draw][circle](2) at (0,3) {1};
\node[draw][circle](6) at (1,3) {5};
\node[draw][circle](10) at (2,3) {9};

\node[draw][circle](3) at (0,2) {2};
\node[draw][circle](7) at (1,2) {6};
\node[draw][circle,inner sep=2pt](11) at (2,2) {10};

\node[draw][circle](4) at (0,1) {3};
\node[draw][circle](8) at (1,1) {7};
\node[draw][circle,inner sep=2pt](12) at (2,1) {11};

\node (13) at (3,4) {} ;
\node (14) at (3,3) {};
\node (15) at (3,2) {};
\node (16) at (3,1) {};
\node at (1,0) {$N_{1,4}(n)$};


\draw[dotted]
(1)--(2)--(3)--(4)
(4)--(5)--(6)--(7)
(7)--(8)--(9)--(10)
(1)--(5)--(9)--(11)
(2)--(6)--(10)--(14)
(3)--(7)--(11)--(15)
(4)--(8)--(12)--(16)
(4)--(5)
(8)--(9)
(12)--(13);

\draw[very thick]
(5)--(1)--(2);


\def\x{5}

\node[draw][circle](1) at (\x+0,4) {0};
\node[draw][circle](5) at (\x+1,4) {4};
\node[draw][circle](9) at (\x+2,4) {8};

\node[draw][circle](2) at (\x+0,3) {1};
\node[draw][circle](6) at (\x+1,3) {5};
\node[draw][circle](10) at (\x+2,3) {9};

\node[draw][circle](3) at (\x+0,2) {2};
\node[draw][circle](7) at (\x+1,2) {6};
\node[draw][circle,inner sep=2pt](11) at (\x+2,2) {10};

\node[draw][circle](4) at (\x+0,1) {3};
\node[draw][circle](8) at (\x+1,1) {7};
\node[draw][circle,inner sep=2pt](12) at (\x+2,1) {11};

\node (13) at (\x+3,4) {} ;
\node (14) at (\x+3,3) {};
\node (15) at (\x+3,2) {};
\node (16) at (\x+3,1) {};
\node at (\x+1,0) {$N_{1,4}^{(2)}(n)$};
\node at (\x*0.5+1,0) {$=$};


\draw[dotted]
(1)--(2)--(3)--(4)
(4)--(5)--(6)--(7)
(7)--(8)--(9)--(10)
(1)--(5)--(9)--(11)
(2)--(6)--(10)--(14)
(3)--(7)--(11)--(15)
(4)--(8)--(12)--(16)
(4)--(5)
(8)--(9)
(12)--(13);

\draw[very thick]
(5)--(1)--(2)--(3);


\def\y{9}

\node[draw][circle](1) at (\y+0,4) {0};
\node[draw][circle](5) at (\y+1,4) {4};
\node[draw][circle](9) at (\y+2,4) {8};

\node[draw][circle](2) at (\y+0,3) {1};
\node[draw][circle](6) at (\y+1,3) {5};
\node[draw][circle](10) at (\y+2,3) {9};

\node[draw][circle](3) at (\y+0,2) {2};
\node[draw][circle](7) at (\y+1,2) {6};
\node[draw][circle,inner sep=2pt](11) at (\y+2,2) {10};

\node[draw][circle](4) at (\y+0,1) {3};
\node[draw][circle](8) at (\y+1,1) {7};
\node[draw][circle,inner sep=2pt](12) at (\y+2,1) {11};

\node (13) at (\y+3,4) {} ;
\node (14) at (\y+3,3) {};
\node (15) at (\y+3,2) {};
\node (16) at (\y+3,1) {};
\node at (\y+1,0) {$N_{1,4}^{(3)}(n)$};
\node at (\x*0.5+0.5*\y+1,0) {$+$};


\draw[dotted]
(1)--(2)--(3)--(4)
(4)--(5)--(6)--(7)
(7)--(8)--(9)--(10)
(1)--(5)--(9)--(11)
(2)--(6)--(10)--(14)
(3)--(7)--(11)--(15)
(4)--(8)--(12)--(16)
(4)--(5)
(8)--(9)
(12)--(13);

\draw[very thick]
(5)--(1)--(2)--(6)
(4)--(3)--(7);

\end{tikzpicture}
\end{center}
\caption{Counting $(1,4)$-Necklaces.}
\label{14cases1}
\end{figure}


Any Hamiltonian cycle in $G_{1,4}(n)$ containing the path $4-0-1-2$ contains exactly one of the edges $2-3$ or $2-6$. The first of these scenarios forces $C$ to contain the edge $3-7$, provided $n \ge 6$. In the second case, 3 has neighbours 4 and 7 in $C$ and 5 has neighbours 6 and 9 in $C$, yielding the path $9-5-6-2-1-0-4-3-7$. Denoting the respective Hamiltonian cycle counts in these two cases by $N_{1,4}^{(4)}(n)$ and $N_{1,4}^{(5)}(n)$. we obtain Figure~\ref{14cases2}. 
 

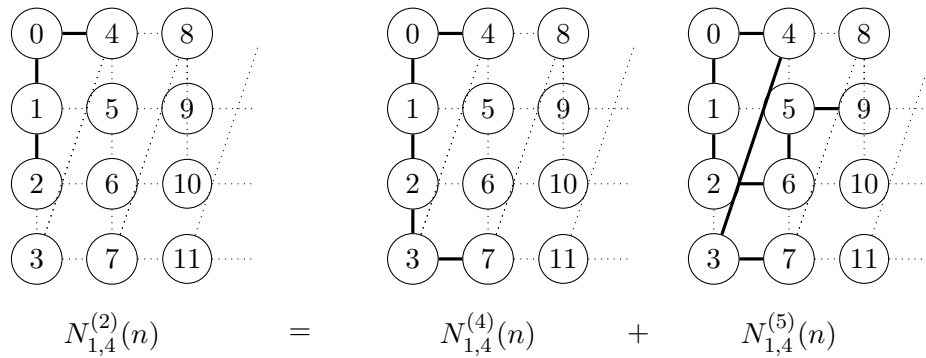
\begin{figure}[h!]
\begin{center}
\begin{tikzpicture}[scale=1]

\node[draw][circle](1) at (0,4) {0};
\node[draw][circle](5) at (1,4) {4};
\node[draw][circle](9) at (2,4) {8};

\node[draw][circle](2) at (0,3) {1};
\node[draw][circle](6) at (1,3) {5};
\node[draw][circle](10) at (2,3) {9};

\node[draw][circle](3) at (0,2) {2};
\node[draw][circle](7) at (1,2) {6};
\node[draw][circle,inner sep=2pt](11) at (2,2) {10};

\node[draw][circle](4) at (0,1) {3};
\node[draw][circle](8) at (1,1) {7};
\node[draw][circle,inner sep=2pt](12) at (2,1) {11};

\node (13) at (3,4) {} ;
\node (14) at (3,3) {};
\node (15) at (3,2) {};
\node (16) at (3,1) {};
\node at (1,0) {$N_{1,4}^{(2)}(n)$};


\draw[dotted]
(1)--(2)--(3)--(4)
(4)--(5)--(6)--(7)
(7)--(8)--(9)--(10)
(1)--(5)--(9)--(11)
(2)--(6)--(10)--(14)
(3)--(7)--(11)--(15)
(4)--(8)--(12)--(16)
(4)--(5)
(8)--(9)
(12)--(13);

\draw[very thick]
(5)--(1)--(2)--(3);


\def\x{5}

\node[draw][circle](1) at (\x+0,4) {0};
\node[draw][circle](5) at (\x+1,4) {4};
\node[draw][circle](9) at (\x+2,4) {8};

\node[draw][circle](2) at (\x+0,3) {1};
\node[draw][circle](6) at (\x+1,3) {5};
\node[draw][circle](10) at (\x+2,3) {9};

\node[draw][circle](3) at (\x+0,2) {2};
\node[draw][circle](7) at (\x+1,2) {6};
\node[draw][circle,inner sep=2pt](11) at (\x+2,2) {10};

\node[draw][circle](4) at (\x+0,1) {3};
\node[draw][circle](8) at (\x+1,1) {7};
\node[draw][circle,inner sep=2pt](12) at (\x+2,1) {11};

\node (13) at (\x+3,4) {} ;
\node (14) at (\x+3,3) {};
\node (15) at (\x+3,2) {};
\node (16) at (\x+3,1) {};
\node at (\x+1,0) {$N_{1,4}^{(4)}(n)$};
\node at (\x*0.5+1,0) {$=$};


\draw[dotted]
(1)--(2)--(3)--(4)
(4)--(5)--(6)--(7)
(7)--(8)--(9)--(10)
(1)--(5)--(9)--(11)
(2)--(6)--(10)--(14)
(3)--(7)--(11)--(15)
(4)--(8)--(12)--(16)
(4)--(5)
(8)--(9)
(12)--(13);

\draw[very thick]
(5)--(1)--(2)--(3)--(4)--(8);


\def\y{9}

\node[draw][circle](1) at (\y+0,4) {0};
\node[draw][circle](5) at (\y+1,4) {4};
\node[draw][circle](9) at (\y+2,4) {8};

\node[draw][circle](2) at (\y+0,3) {1};
\node[draw][circle](6) at (\y+1,3) {5};
\node[draw][circle](10) at (\y+2,3) {9};

\node[draw][circle](3) at (\y+0,2) {2};
\node[draw][circle](7) at (\y+1,2) {6};
\node[draw][circle,inner sep=2pt](11) at (\y+2,2) {10};

\node[draw][circle](4) at (\y+0,1) {3};
\node[draw][circle](8) at (\y+1,1) {7};
\node[draw][circle,inner sep=2pt](12) at (\y+2,1) {11};

\node (13) at (\y+3,4) {} ;
\node (14) at (\y+3,3) {};
\node (15) at (\y+3,2) {};
\node (16) at (\y+3,1) {};
\node at (\y+1,0) {$N_{1,4}^{(5)}(n)$};
\node at (\x*0.5+0.5*\y+1,0) {$+$};


\draw[dotted]
(1)--(2)--(3)--(4)
(4)--(5)--(6)--(7)
(7)--(8)--(9)--(10)
(1)--(5)--(9)--(11)
(2)--(6)--(10)--(14)
(3)--(7)--(11)--(15)
(4)--(8)--(12)--(16)
(4)--(5)
(8)--(9)
(12)--(13);

\draw[very thick]
(8)--(4)--(5)--(1)--(2)--(3)--(7)--(6)--(10);

\end{tikzpicture}
\end{center}

\caption{Counting $(1,4)$-necklaces containing the path $4-0-1-2$.}
\label{14cases2}
\end{figure}


Similarly, if $C$ contains the paths $4-0-1-5$ and $3-2-6$, then it contains either the edge $3-4$ or the edge $3-7$. The former situation yields the path $6-2-3-4-0-1-5$ in $C$, and 5 must be adjacent to~9 if $n \ge 8$. If $C$ contains $4-0-1-5$ and $7-3-2-6$, then 4 must be adjacent to 8 in $C$ to avoid forming the small cycle $4-0-1-5$. This yields the decomposition of $N_{1,4}^{(3)}(n)$ shown in Figure~\ref{14cases3}. 


\begin{figure}[h!]
\begin{center}
\begin{tikzpicture}[scale=1]

\node[draw][circle](1) at (0,4) {0};
\node[draw][circle](5) at (1,4) {4};
\node[draw][circle](9) at (2,4) {8};

\node[draw][circle](2) at (0,3) {1};
\node[draw][circle](6) at (1,3) {5};
\node[draw][circle](10) at (2,3) {9};

\node[draw][circle](3) at (0,2) {2};
\node[draw][circle](7) at (1,2) {6};
\node[draw][circle,inner sep=2pt](11) at (2,2) {10};

\node[draw][circle](4) at (0,1) {3};
\node[draw][circle](8) at (1,1) {7};
\node[draw][circle,inner sep=2pt](12) at (2,1) {11};

\node (13) at (3,4) {} ;
\node (14) at (3,3) {};
\node (15) at (3,2) {};
\node (16) at (3,1) {};
\node at (1,0) {$N_{1,4}^{(3)}(n)$};


\draw[dotted]
(1)--(2)--(3)--(4)
(4)--(5)--(6)--(7)
(7)--(8)--(9)--(10)
(1)--(5)--(9)--(11)
(2)--(6)--(10)--(14)
(3)--(7)--(11)--(15)
(4)--(8)--(12)--(16)
(4)--(5)
(8)--(9)
(12)--(13);

\draw[very thick]
(5)--(1)--(2)--(6)
(4)--(3)--(7);


\def\x{5}

\node[draw][circle](1) at (\x+0,4) {0};
\node[draw][circle](5) at (\x+1,4) {4};
\node[draw][circle](9) at (\x+2,4) {8};

\node[draw][circle](2) at (\x+0,3) {1};
\node[draw][circle](6) at (\x+1,3) {5};
\node[draw][circle](10) at (\x+2,3) {9};

\node[draw][circle](3) at (\x+0,2) {2};
\node[draw][circle](7) at (\x+1,2) {6};
\node[draw][circle,inner sep=2pt](11) at (\x+2,2) {10};

\node[draw][circle](4) at (\x+0,1) {3};
\node[draw][circle](8) at (\x+1,1) {7};
\node[draw][circle,inner sep=2pt](12) at (\x+2,1) {11};

\node (13) at (\x+3,4) {} ;
\node (14) at (\x+3,3) {};
\node (15) at (\x+3,2) {};
\node (16) at (\x+3,1) {};
\node at (\x+1,0) {$N_{1,4}^{(6)}(n)$};
\node at (\x*0.5+1,0) {$=$};


\draw[dotted]
(1)--(2)--(3)--(4)
(4)--(5)--(6)--(7)
(7)--(8)--(9)--(10)
(1)--(5)--(9)--(11)
(2)--(6)--(10)--(14)
(3)--(7)--(11)--(15)
(4)--(8)--(12)--(16)
(4)--(5)
(8)--(9)
(12)--(13);

\draw[very thick]
(7)--(3)--(4)--(5)--(1)--(2)--(6)--(10);


\def\y{9}

\node[draw][circle](1) at (\y+0,4) {0};
\node[draw][circle](5) at (\y+1,4) {4};
\node[draw][circle](9) at (\y+2,4) {8};

\node[draw][circle](2) at (\y+0,3) {1};
\node[draw][circle](6) at (\y+1,3) {5};
\node[draw][circle](10) at (\y+2,3) {9};

\node[draw][circle](3) at (\y+0,2) {2};
\node[draw][circle](7) at (\y+1,2) {6};
\node[draw][circle,inner sep=2pt](11) at (\y+2,2) {10};

\node[draw][circle](4) at (\y+0,1) {3};
\node[draw][circle](8) at (\y+1,1) {7};
\node[draw][circle,inner sep=2pt](12) at (\y+2,1) {11};

\node (13) at (\y+3,4) {} ;
\node (14) at (\y+3,3) {};
\node (15) at (\y+3,2) {};
\node (16) at (\y+3,1) {};
\node at (\y+1,0) {$N_{1,4}^{(7)}(n)$};
\node at (\x*0.5+0.5*\y+1,0) {$+$};


\draw[dotted]
(1)--(2)--(3)--(4)
(4)--(5)--(6)--(7)
(7)--(8)--(9)--(10)
(1)--(5)--(9)--(11)
(2)--(6)--(10)--(14)
(3)--(7)--(11)--(15)
(4)--(8)--(12)--(16)
(4)--(5)
(8)--(9)
(12)--(13);

\draw[very thick]
(9)--(5)--(1)--(2)--(6)
(8)--(4)--(3)--(7);

\end{tikzpicture}
\end{center}

\caption{Counting $(1,4)$-necklaces containing the path $4-0-1-5$.}
\label{14cases3}
\end{figure}
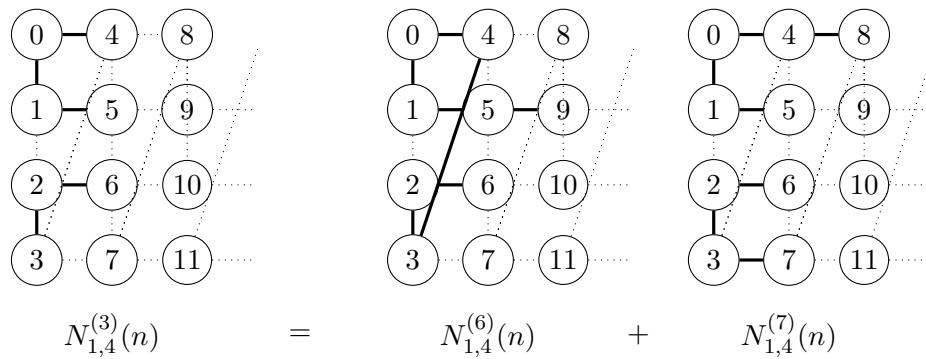

Applying Lemma~\ref{downshift} to the Hamiltonian cycles in $G_{1,4}(n)$ counted by $N_{1,4}^{(i)}(n)$ for \mbox{$i = 4, 5, 6, 7$} yields
\begin{align*}
    N_{1,4}^{(4)}(n) &= N_{1,4}(n-3) \ , \\
    N_{1,4}^{(5)}(n) &= N_{1,4}^{(2)}(n-5) \ , \\
    N_{1,4}^{(6)}(n) &= N_{1,4}(n-5) \ , \\
    N_{1,4}^{(7)}(n) &= N_{1,4}^{(3)}(n-2) \ , 
\end{align*}
where for $N_{1,4}^{(7)}(n)$, the construction of $S_2(C)$ is applied simultaneously to the path $8-4-0-1-5$ which becomes $6-2-3$, and the path $7-3-2-6$ also contained in $C$ becomes $5-1-0-4$ in $S_2(C)$. 
Solving for $N_{1,4}(n)$ produces the linear recurrence 
\begin{align}\label{14rec}
N_{1,4}(n) & = N_{1,4}(n-3)+N_{1,4}(n-5)+N_{1,4}(n-7)+N_{1,4}(n-8)\nonumber \\
           & \qquad +N_{1,4}(n-9)+N_{1,4}(n-10)+N_{1,4}(n-11)+N_{1,4}(n-13)  
\end{align}
for $n\geq 16$. The first 13 terms can be obtained by hand and/or computer search and are listed in Table~\ref{14terms}.

\begin{table} \centering
 \begin{tabular}{| c || c | c |c |c |c |c |c |c |c |c |c |} 
 \hline
          $n$  & $5$ & $6$ & $7$ & $8$ & $9$ & $10$ & $11$ & $12$ & $13$ & $14$ & $15$\\ 
 \hline
 $N_{1,4}(n) $ & $1$ & $0$ & $1$ & $1$ & $1$ & $3$  & $2$  & $3$  & $6$  & $5$  & $10$ \\ 
 \hline
\end{tabular}
\caption{Initial values of $N_{1,4}(n)$.}
\label{14terms}
\end{table}

We remark that the characteristic polynomial corresponding to the recurrence in~(\ref{14rec}) can be factored into a product of two irreducible polynomials of degree 4 and 9, thereby giving the following lower order recurrence for $N_{1,4}(n)$:
\begin{align*}
N_{1,4}(n) & = -N_{1,4}(n-1) + N_{1,4}(n-3) + N_{1,4}(n-4) + 2N_{1,4}(n-5) \\
           & \qquad + 2N_{1,4}(n-6) + N_{1,4}(n-7) + N_{1,4}(n-8) + N_{1,4}(n-9) \ .
\end{align*}
Here, order~9 is optimal in the sense that $N_{1,4}(n)$ does not obey a linear recurrence of order~8 or less.


The downshifting technique employed in Lemma \ref{downshift} can in principle be used to count $(a,b)$-necklaces for any permissible pair $(a,b)$. The approach can be generalized considerably; we already encountered a version applied to Hamiltonian cycles with additional graph containment restrictions in the derivation of $N_{1,4}^{(7)}(n)$ and a variant allowing suitable permutations of degree~2 vertices in the determination of $N_{(2,3)}^{(2)}(n)$. The main obstacle in applying the method to larger parameters $(a,b)$ is a combinatorial explosion of the number of cases partitioning  Hamiltonian cycles according to suitable subgraph configurations. Although it is difficult to enumerate $(a,b)$-necklaces for specific values $a,b$, we prove in the next section that the count always satisfies a linear homogeneous recurrence relation.

\section{Necklaces as Walks in a Weighted Digraph} \label{S:general}

Throughout this section, fix positive coprime integers $a, b$ with $a < b$. The techniques used in the previous section can be generalized to prove that $N_{a,b}(n)$ obeys a linear homogeneous recurrence relation of fixed order with integer coefficients. To that end, we will construct a weighted digraph $D_{a,b}$ in which certain walks whose edge weights sum to $n$ are in one-to-one correspondence with $(a,b)$-necklaces of length~$n$. Lemma~\ref{recur} shows that the number of such walks satisfies a linear recurrence; hence, so does the count $N_{a,b}(n)$ of $(a,b)$-necklaces. 

Let $D$ be a weighted directed graph with a weight function $w$ defined on its edges. For any walk $W$ in $D$ consisting of edges $e_1, e_2, \ldots , e_k$, define
\begin{equation} \label{weightsum}
   S(W) = \sum_{i=1}^k w(e_i) 
\end{equation} 
to be the sum of the edge weights of $P$. 

\begin{lemma} \label{recur}
Let $D$ be weighted directed graph. For $n \in \mathbb{Z}^{> 0}$, let $f(n)$ denote the number of walks~$W$ in $D$ such that $S(W) = n$. Then $f(n)$ satisfies a linear homogeneous recurrence relation.
\end{lemma}
\begin{proof}
Construct an unweighted digraph $D^*$ from $D$ by subdividing every edge $e = (u,v)$ of weight $m$ in $D$ into a directed unweighted path of length $m$ from $u$ to $v$ in $D^*$. Then for any pair of vertices $x, y$ in $D$ and any $n > 0$, the walks $W$ from $x$ to $y$ in $D$ with $S(W) = n$ are in one-to-one correspondence with the walks of length $n$ from $x$ to $y$ in $D^*$. It is well-known that in a directed graph, the number of walks of fixed length connecting any pair of vertices obeys a linear homogeneous recurrence; see Theorem~4.7.1 and Corollary~4.7.4 of~\cite{RS}. 
\end{proof}

For example of this construction, see Figure~\ref{sumdigraph}.

\begin{figure}[h!]
    \centering
    \begin{subfigure}[t]{0.5\textwidth}
        \centering
     \begin{tikzpicture}[scale=0.6]
     
\node[draw][circle](1) at (0,3) {$A$};
\node[draw][circle](2) at (-3,0) {$B$};
\node[draw][circle](3) at (3,0) {$C$};

\path[->,thick]
(1) edge[bend right] node [left] {} (2)
(2) edge[bend right] node [left] {} (1)
(2) edge[bend right] node [left] {} (3)
(3) edge[bend right] node [left] {} (1);

\draw
(2.4,2.4) node{$1$}
(0,-0.5) node{$3$}
(-1,1.4) node{$1$}
(-2.4,2.4) node{$2$};

     \end{tikzpicture}
        \caption{ $D$}
    \end{subfigure}%
    ~ 
    \begin{subfigure}[t]{0.5\textwidth}
        \centering
     \begin{tikzpicture}[scale=0.6]
\node[draw][circle](1) at (0,3) {$A$};
\node[draw][circle](2) at (-3,0) {$B$};
\node[draw][circle](3) at (3,0) {$C$};
\node[draw][circle, inner sep=3pt](4) at (-1,-1) {};
\node[draw][circle, inner sep=3pt](5) at (1,-1) {};
\node[draw][circle, inner sep=3pt](6) at (-2.4,2.4) {};

\path[->,thick]
(2) edge[bend right] node [left] {} (4)
(4) edge[right] node [left] {} (5)
(5) edge[bend right] node [left] {} (3)
(2) edge[bend right] node [left] {} (1)
(3) edge[bend right] node [left] {} (1)
(1) edge[bend right] node [left] {} (6)
(6) edge[bend right] node [left] {} (2);

    \end{tikzpicture}
        \caption{$D^\ast$}
    \end{subfigure}
    \caption{}
    \label{sumdigraph}
\end{figure}

Our approach to establishing a linear recurrence for the count of $(a,b)$-necklaces via the aforementioned digraph construction draws on the \emph{transfer matrix method} described in Section~4.7 of~\cite{RS}. This technique has been deployed for enumerating Hamiltonian cycles in a variety of families of graphs; see the survey~\cite{G} for the application to grid, cylindrical, and torus graphs, for example. 

In our construction of $D_{a,b}$, we make extensive use of the topology of the ragged rectangular depiction of the graph $G_{a,b}(n)$. Write
\begin{equation} \label{ndiv}
n-1 = qb + r, \quad 0 \leq r < b, \quad q = \left \lfloor \frac{n-1}{b} \right \rfloor . 
\end{equation}
Then $G_{a,b}(n)$ has $b$ rows, labeled $0, 1, \ldots , b-1$, and $q+1$ columns, labeled $0, 1, \ldots , q$. The number of vertices in any column of $G_{a,b}(n)$ is referred to as its \emph{length}. Columns $0, 1, \ldots , q-1$ all have length $b$, while column $q$ has length $r+1 \le b$. Row~$i$ and column~$j$ intersect at vertex $jb+i$. It is easy to see that every edge in $G_{a,b}(n)$ joins two vertices in either the same column or adjacent columns. We use the same labeling and terminology for any subgraph of $G_{a,b}$.

\subsection{Construction for 2-Regular Spanning Subgraphs of $G_{a,b}$} \label{D'ab}

The description of the aforementioned directed graph $D_{a,b}$ is rather technical, so we first present a significantly simpler construction of a related digraph~$D'_{a,b}$ in which walks $W$ with $S(W) = n$ correspond to 2-regular spanning subgraphs (i.e.\ disjoint unions of cycles on all the vertices) of~$G_{a,b}(n)$. The idea underlying both constructions is to decompose any such subgraph in $G_{a,b}(n)$, with $n>2b$, into a unique sequence of smaller subgraphs on pairs of consecutive columns, called \emph{blocks}. For example, Figure~\ref{HC23} depicts Hamiltonian cycles in $G_{2,3}(14)$ and $G_{2,3}(17)$. The constituent blocks of these two Hamiltonian cycles are pictured in~ Figure~\ref{blocks}. Up to isomorphism, the two Hamiltonian cycles can be reconstructed through ``gluing'' together the appropriate blocks by identifying the vertices in the right column of a block with those in the left column of the next block in the sequence. Specifically, the Hamiltonian cycle in Figure~\ref{HC23}(a) is obtained in this way from the blocks $B_1, B_2, B_3, B_4$, and the Hamiltonian cycle in Figure~\ref{HC23}(b) from the sequence $B_1, B_2, B_2, B_3, B_4$. Note that the edges $6-8$ in both graphs and the edge $9-11$ in $G_{2,3}(17)$ are omitted from the left column of block~$B_3$ but included in the right column of~$B_2$. In general, edges joining two vertices in the left column of any block other than the first block will removed. 

\begin{figure}[h!]
    \centering
    \begin{subfigure}[t]{0.5\textwidth}
        \centering
     \begin{tikzpicture}[scale=0.6]
     
\node[draw][circle](1) at (0,3) {0};
\node[draw][circle](2) at (0,1.5) {1};
\node[draw][circle](3) at (0,0) {2};

\node[draw][circle](4) at (2,3) {3};
\node[draw][circle](5) at (2,1.5) {4};
\node[draw][circle](6) at (2,0) {5};

\node[draw][circle](7) at (4,3) {6};
\node[draw][circle](8) at (4,1.5) {7};
\node[draw][circle](9) at (4,0) {8};

\node[draw][circle](10) at (6,3) {9};
\node[draw][circle, inner sep=2pt](11) at (6,1.5) {10};
\node[draw][circle, inner sep=2pt](12) at (6,0) {11};

\node[draw][circle, inner sep=2pt](13) at (8,3) {12};
\node[draw][circle, inner sep=2pt](14) at (8,1.5) {13};

\draw 
(1)--(4)--(2)--(5)--(7)
(9)--(12)--(14)--(11)--(13)--(10)--(8)--(6)--(3)

(-0.45,2.7) arc (130:230:1.5)
(3.55,2.7) arc (130:230:1.5);

     \end{tikzpicture}
        \caption{ Hamiltonian cycle in $G_{2,3}(14)$.}
    \end{subfigure}%
    ~ 
    \begin{subfigure}[t]{0.5\textwidth}
        \centering
     \begin{tikzpicture}[scale=0.6]
\node[draw][circle](1) at (0,3) {0};
\node[draw][circle](2) at (0,1.5) {1};
\node[draw][circle](3) at (0,0) {2};

\node[draw][circle](4) at (2,3) {3};
\node[draw][circle](5) at (2,1.5) {4};
\node[draw][circle](6) at (2,0) {5};

\node[draw][circle](7) at (4,3) {6};
\node[draw][circle](8) at (4,1.5) {7};
\node[draw][circle](9) at (4,0) {8};

\node[draw][circle](10) at (6,3) {9};
\node[draw][circle, inner sep=2pt](11) at (6,1.5) {10};
\node[draw][circle, inner sep=2pt](12) at (6,0) {11};

\node[draw][circle, inner sep=2pt](13) at (8,3) {12};
\node[draw][circle, inner sep=2pt](14) at (8,1.5) {13};
\node[draw][circle, inner sep=2pt](15) at (8,0) {14};

\node[draw][circle, inner sep=2pt](16) at (10,3) {15};
\node[draw][circle, inner sep=2pt](17) at (10,1.5) {16};

\draw 
(8)--(6)--(3)
(1)--(4)--(2)--(5)--(7)
(8)--(10)
(12)--(15)--(17)--(14)--(16)--(13)--(11)--(9)

(-0.45,2.7) arc (130:230:1.5)
(3.55,2.7) arc (130:230:1.5)
(5.55,2.7) arc (130:230:1.5);

    \end{tikzpicture}
        \caption{Hamiltonian cycle in $G_{2,3}(17)$.}
    \end{subfigure}
    \caption{}
    \label{HC23}
\end{figure}

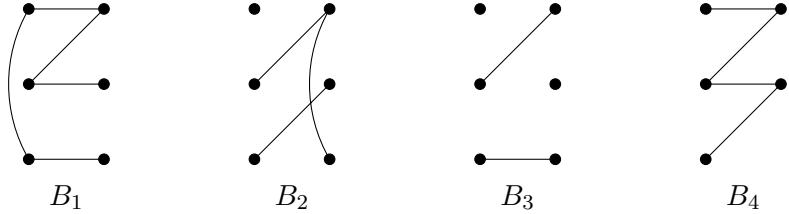
\begin{figure}[h!]
\begin{center}
\begin{tikzpicture}[scale=1]

\def\x{3}
\def\y{6}
\def\z{9}

\filldraw

(0,0) circle (2pt)
(0,1) circle (2pt)
(0,2) circle (2pt)

(1,0) circle (2pt)
(1,1) circle (2pt)
(1,2) circle (2pt)

(\x,0) circle (2pt)
(\x,1) circle (2pt)
(\x,2) circle (2pt)

(1+\x,0) circle (2pt)
(1+\x,1) circle (2pt)
(1+\x,2) circle (2pt)

(0+\y,0) circle (2pt)
(0+\y,1) circle (2pt)
(0+\y,2) circle (2pt)

(1+\y,0) circle (2pt)
(1+\y,1) circle (2pt)
(1+\y,2) circle (2pt)

(\z,0) circle (2pt)
(\z,1) circle (2pt)
(\z,2) circle (2pt)

(1+\z,1) circle (2pt)
(1+\z,2) circle (2pt);

\draw

(0,2) arc (150:210:2)
(0,2)--(1,2)--(0,1)--(1,1)
(0,0)--(1,0);

\draw

(\x,1)--(\x+1,2)
(\x+1,2) arc (150:210:2)
(\x,0)--(\x+1,1);

\draw

(\y,1)--(\y+1,2)
(\y,0)--(\y+1,0);

\draw 

(\z,0)--(\z+1,1)--(\z,1)--(\z+1,2)--(\z,2);

\draw

(0.5,-0.5) node{$B_1$} 
(3.5,-0.5) node{$B_2$} 
(6.5,-0.5) node{$B_3$} 
(9.5,-0.5) node{$B_4$} ;

\end{tikzpicture}
\end{center}
\caption{The constituent blocks of the Hamiltonian cycles in Figure~\ref{HC23}.}
\label{blocks}
\end{figure}

For every $n > b$, any two adjacent columns of $G_{a,b}(n)$ induce a subgraph that is isomorphic to $G_{a,b}(b+r+1)$ if the right column is the last column (column $q$) of $G_{a,b}(n)$ and to $G_{a,b}(2b)$ otherwise. Up to isomorphism, a block is a subgraph of the two-column graph $G_{a,b}(2b)$  whose maximum vertex degree is~2, subject to certain additional conditions. We formalize this notion in Definition~\ref{decomp2renate}.

\begin{defn}\label{decomp2renate} 
An $(a,b)$-block $B$ is a graph that is isomorphic to a spanning subgraph of $G_{a,b}(b+s)$ such that every vertex in $B$ has degree at most~2, subject to exactly one of following additional conditions: 
\begin{itemize} \itemsep 0pt
    \item $s = b$ and every vertex in the left column of $B$ has degree~2. In this case, $B$ is a \emph{start block}. 
    \item $s = b$, the right column of $B$ contains at least one vertex of degree $\le 1$, and no two vertices in the left column of $B$ are adjacent. In this case, $B$ is \emph{mid block}.  
    \item $1 \le s \le b$, every vertex in the right column of $B$ has degree~2, and no two vertices in the left column of $B$ are adjacent. In this case $B$ is an \emph{end block}.
\end{itemize}
For an ordered sequence of (not necessarily distinct) $(a,b)$-blocks $\mathbf{B} = (B_1, B_2, \ldots , B_m)$ such that $B_1, B_2, \ldots , B_{m-1}$ all have right columns of length $b$, define $\mathcal{G}(\mathbf{B})$ to be the graph obtained by identifying the $i$-th vertex the right column of $B_j$ with the $i$-th vertex in the left column of $B_{j+1}$, for $0 \leq i \leq b-1$ and $1 \leq j \le m-1$. 
\end{defn}

For example, in Figure~\ref{blocks}, $B_1$ is a $(2,3)$-start block, $B_2$ and $B_3$ are  $(2,3)$-mid blocks, and $B_4$ is a  $(2,3)$-end block. The Hamiltonian cycles in Figures~\ref{HC23}(a) and~\ref{HC23}(b) are isomorphic to $\mathcal{G}(B_1, B_2, B_3, B_4)$ and $\mathcal{G}(B_1, B_2, B_2, B_3, B_4)$, respectively. The graph $\mathcal{G}(B_1, B_3, B_4)$ is isomorphic to a Hamiltonian cycle of $G_{2,3}(11)$. 

When the values of $a$ and $b$ are clear from context, we will refer to an $(a,b)$-block as simply a block. Note that every block has a left column of length $b$ and a right column of length at most $b$, and only end blocks may have a right column of length less than~$b$. Only start blocks may contain edges joining vertices in their left column. Moreover, for any $n > b$, the only neighbours of $0$ in $G_{a,b}(n)$ are $a$ in column~0 and $b$ in column~1. Hence, the top left corner vertex of any mid block or end block has degree at most~1.

In order to obtain a Hamiltonian cycle from a sequence of blocks, adjacent blocks must fit together in such a way that every vertex in their shared column has degree~2. This is captured in the notion of compatibility.

\begin{defn}
Let $B, B'$ be two (not necessarily distinct) blocks. Then the ordered pair $(B, B')$ is \emph{compatible} if the right column of $B$ has length~$b$ and every vertex in the middle column of $\mathcal{G}(B, B')$ has degree 2. A finite ordered sequence of blocks $(B_1, B_2, \ldots , B_m)$ is compatible if $B_1$ is a start block and $(B_j, B_{j+1})$ is compatible for $1 \le j \le m-1$.
\end{defn}

In Figure~\ref{blocks}, the pairs of compatible blocks are $(B_1, B_2)$, $(B_1, B_3)$, $(B_2, B_2)$, $(B_2, B_3)$ and $(B_3, B_4)$. Examples of non-compatible block pairs include $(B_1, B_4)$ and $(B_3, B_2)$, since the middle column of $\mathcal{G}(B_1, B_4)$ and $\mathcal{G}(B_3, B_2)$ contain vertices of degree~3 and~1, respectively. A family of compatible sequences is given by $(B_1, B_2, \ldots , B_2, B_3, B_4)$, with zero or more occurrences of $B_2$. 

\begin{lemma}\label{usecompatible}
Let $q \ge 2$ and let $\mathbf{B} = (B_1, B_2, \ldots , B_q)$ be a sequence of compatible blocks such that $B_q$ is an end block. Put $n = qb + s$ where $s$ is the length of the right column of $B_q$. Then $\mathcal{G}(\mathbf{B})$ is isomorphic to a $2$-regular subgraph of $G_{a,b}(n)$ on $n$ vertices. 
\end{lemma}
\begin{proof}
The graph $\mathcal{G}(\mathbf{B})$ has $qb+s = n$ vertices and is 2-regular since $\mathbf{B}$ is compatible. 
\end{proof}

\begin{lemma} \label{decomp22} 
Fix $n > 2b$, let $q, r$ be as defined in \eqref{ndiv} and let $H$ be a $2$-regular subgraph of $G_{a,b}(n)$ on $n$ vertices. Then there exists a unique compatible sequence $\mathbf{B} = (B_1, B_2, \ldots , B_q)$ of blocks, with $B_q$ an end block whose right column has length $r+1$, such that $H$ is isomorphic to $\mathcal{G}(\mathbf{B})$. Furthermore, if $H$ is a Hamiltonian cycle of $G_{a,b}(n)$, then $B_2, B_3 \ldots , B_{q-2}$ are mid blocks. 
\end{lemma}
\begin{proof}
Observe that $H$ has at least three columns, since $n > 2b$ forces $q \ge 2$. Let $B_1$ be the subgraph induced by the first two columns (columns~0 and~1) of $H$, and for $2 \le j \le q$, let $B_j$ be the subgraph induced by columns~$j-1$ and~$j$ of~$H$, with every edge joining two vertices in column $j-1$ removed. Since~$H$ is 2-regular, it is clear that $\mathbf{B}$ is a compatible sequence of blocks such that $B_1$ is a start block, $B_q$ is an end block whose right column has length $r+1$, and $H$ is isomorphic to $\mathcal{G}(\mathbf{B})$. Furthermore, $\mathbf{B}$ is the only compatible sequence of blocks with these properties. 

Suppose that $H$ is a Hamiltonian cycle. Let $B_j$ be the first end block in $\mathbf{B}$. Then $\mathcal{G}(B_1, B_2, \ldots , B_j)$ is a 2-regular subgraph of $H$ by Lemma~\ref{usecompatible}. Since $H$ is connected, this subgraph must be all of $H$, so $j = q$. Similarly, let $B_k$ be the last start block in $\mathbf{B}$. Then again by Lemma~\ref{usecompatible}, $\mathcal{G}(B_k, B_{k+1}, \ldots , B_q)$ is a 2-regular subgraph of $H$, forcing $k = 1$. It follows that none of $B_2, B_3 \ldots , B_{q-2}$ is a start block or an end block, so they are all mid blocks. 
\end{proof}

The converse of the second statement of Lemma~\ref{decomp22} is not true, as sequences consisting of a start block, zero or more mid blocks and an end block may produce disconnected graphs. Figure~\ref{14blocks} shows eight $(1,4)$-blocks, where $B_1, B_2$ are start blocks, $B_3, B_4, B_5$ are mid blocks and $B_6, B_7, B_8$ are end blocks. The graph 
$\mathcal{G}(B_1, B_4, B_6)$ is isomorphic to the disjoint union of two 8-cycles as shown in Figure~\ref{BL}(c). 

\begin{figure}[h!]
\begin{center}
\begin{tikzpicture}[scale=0.75]

\def\u{2.5}
\def\v{5}
\def\w{7.5}
\def\x{10}
\def\y{12.5}
\def\z{15}
\def\zz{17.5}

\filldraw

(0,0) circle (2pt)
(0,1) circle (2pt)
(0,2) circle (2pt)
(0,3) circle (2pt)

(1,0) circle (2pt)
(1,1) circle (2pt)
(1,2) circle (2pt)
(1,3) circle (2pt)

(\u,0) circle (2pt)
(\u,1) circle (2pt)
(\u,2) circle (2pt)
(\u,3) circle (2pt)

(\u+1,0) circle (2pt)
(\u+1,1) circle (2pt)
(\u+1,2) circle (2pt)
(\u+1,3) circle (2pt)
(\v,0) circle (2pt)
(\v,1) circle (2pt)
(\v,2) circle (2pt)
(\v,3) circle (2pt)

(\v+1,0) circle (2pt)
(\v+1,1) circle (2pt)
(\v+1,2) circle (2pt)
(\v+1,3) circle (2pt)
%

%
(\w,0) circle (2pt)
(\w,1) circle (2pt)
(\w,2) circle (2pt)
(\w,3) circle (2pt)

(\w+1,0) circle (2pt)
(\w+1,1) circle (2pt)
(\w+1,2) circle (2pt)
(\w+1,3) circle (2pt)
(\x,0) circle (2pt)
(\x,1) circle (2pt)
(\x,2) circle (2pt)
(\x,3) circle (2pt)

(\x+1,0) circle (2pt)
(\x+1,1) circle (2pt)
(\x+1,2) circle (2pt)
(\x+1,3) circle (2pt)
(\y,0) circle (2pt)
(\y,1) circle (2pt)
(\y,2) circle (2pt)
(\y,3) circle (2pt)

(\y+1,0) circle (2pt)
(\y+1,1) circle (2pt)
(\y+1,2) circle (2pt)
(\y+1,3) circle (2pt)

(\z,0) circle (2pt)
(\z,1) circle (2pt)
(\z,2) circle (2pt)
(\z,3) circle (2pt)

(\z+1,0) circle (2pt)
(\z+1,1) circle (2pt)
(\z+1,2) circle (2pt)
(\z+1,3) circle (2pt)

(\zz,0) circle (2pt)
(\zz,1) circle (2pt)
(\zz,2) circle (2pt)
(\zz,3) circle (2pt)

(\zz+1,3) circle (2pt);


\draw

(1,3)--(0,3)--(0,2)--(1,2)
(1,1)--(0,1)--(0,0)--(1,0);

\draw 
(\u+1,3)--(\u+0,3)--(\u+0,0)--(\u+1,0)
(\u+1,1)--(\u+1,2);

\draw 
(\v+1,3)--(\v+0,3)
(\v+0,0)--(\v+1,0);


\draw 
(\w+0,0)--(\w+1,0)
(\w+0,1)--(\w+1,1)
(\w+0,2)--(\w+1,2)
(\w+0,3)--(\w+1,3);

\draw
(\x+0,0)--(\x+1,0)
(\x+0,1)--(\x+1,1)--(\x+1,2)--(\x+0,2)
(\x+0,3)--(\x+1,3);

\draw 
(\y+0,0)--(\y+1,0)--(\y+1,1)--(\y+0,1)
(\y+0,2)--(\y+1,2)--(\y+1,3)--(\y+0,3);

\draw 
(\z+0,0)--(\z+1,0)--(\z+1,3)--(\z+0,3);

\draw
(\zz,0)--(\zz+1,3)--(\zz,3);

\draw

(0.5,-0.5) node{$B_1$} 
(3,-0.5) node{$B_2$} 
(5.5,-0.5) node{$B_3$} 
(8,-0.5) node{$B_4$} 
(10.5,-0.5) node{$B_5$} 
(13,-0.5) node{$B_6$}
(15.5,-0.5) node{$B_7$}
(18,-0.5) node{$B_8$};

\end{tikzpicture}
\end{center}
\caption{Some $(1,4)$ blocks.}
\label{14blocks}
\end{figure}
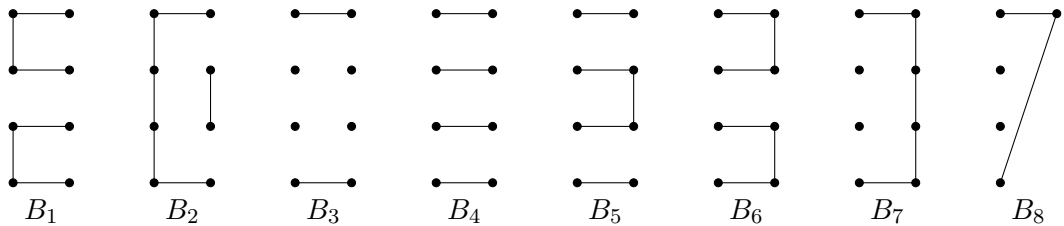

Lemmas~\ref{usecompatible} and~\ref{decomp22} establish that for any $n > 2b$, the 2-regular subgraphs of $G_{a,b}(n)$ are in one-to-one correspondence with the compatible sequences $\mathbf{B} = (B_1, B_2, \ldots , B_q)$ of $(a,b)$-blocks ending in an end block. This suggests the following construction of a weighted digraph $D'_{a,b}$. The vertex set of $D'_{a,b}$ consists of all $(a,b)$-blocks, identified as start, mid or end blocks, along with an additional distinguished start vertex $S$. Note that this is a finite set, since the collection of simple labeled graphs on at most~$2b$ vertices is finite and contains all blocks. The edge set of $D'_{a,b}$ is given as follows:

\begin{enumerate} \itemsep 0pt
    \item Place an edge $(S,B)$ of weight $2b$ joining $S$ to every start block $B$.
    \item For every compatible pair of blocks $(B, \hat{B})$, place an edge $(B, \hat{B})$ of weight $s$, where $s$ is the length of the right column of $\hat{B}$. 
\end{enumerate}

\begin{prop} \label{2regular}
For any $n > 2b$, the $2$-regular spanning subgraphs of $G_{a,b}(n)$ are in one-to-one correspondence with the walks $W$ in $D'_{a,b}$ from $S$ to any end block such that $S(W) = n$, with $S(W)$ as defined in \eqref{weightsum}. 
\end{prop}
\begin{proof}
By Lemmas~\ref{usecompatible} and~\ref{decomp22}, the 2-regular spanning subgraphs of $G_{a,b}(n)$ are in bijection with the compatible sequences $\mathbf{B} = (B_1, B_2, \ldots , B_q)$  of blocks, where the right column of~$B_q$ has length $r+1$, with $r$ as in \eqref{ndiv}. Here, $B_1$ is a start block and $B_q$ is an end block. So every such sequence corresponds bijectively to the walk 
\[ W : S \stackrel{2b}{\longrightarrow} B_1 \stackrel{b}{\longrightarrow} B_2 \stackrel{b}{\longrightarrow} \ldots \stackrel{b}{\longrightarrow} B_{q-1} \stackrel{r+1}{\longrightarrow} B_q \]
in $D'_{a,b}$, and we have $S(W) = 2b + (q-2)b + r+1 = n$.
\end{proof}

Lemma \ref{recur} now immediately yields the following result.

\begin{cor}
The number of 2-regular spanning subgraphs of $G_{a,b}(n)$ satisfies a linear homogeneous recurrence relation.
\end{cor}

\subsection{Construction for Hamiltonian Cycles of $G_{a,b}$} \label{Dab}

The relationship between 2-regular spanning subgraphs $H$ of $G_{a,b}(n)$ and walks $W$  in $D'_{a,b}$ of weight sum $n$ connecting $S$ to an end block is highly constructive: the sequence of vertices in $W$, save its first vertex, is precisely the sequence of blocks constituting $H$. In order to obtain a Hamiltonian cycle $H$, the corresponding walk $W$ in $D'_{a,b}$ must not pass through blocks that result in multiple short cycles. The digraph $D_{a,b}$ will be constructed from $D'_{a,b}$ in a manner that avoids blocks which close such short cycles. Moreover, it will only contain walks in which $S$ is followed by a start block, zero or more mid blocks and an end block, as this is a necessary condition for generating a Hamiltonian cycle by Lemma~\ref{decomp22}. 

Consider for example the two compatible sequences depicted in Figure~\ref{14gluedgraphs}. The pathwise connected pairs of vertices in the rightmost column of~$\mathcal{G}(B_1,B_4)$ belong to different rows compared to those in the rightmost column of $\mathcal{G}(B_2,B_4)$. Appending $B_6$ to either of these graphs introduces two new paths connecting vertices in the rightmost column of the graph, one connecting the vertices in rows~0 and~1, and the other connecting the vertices in rows~2 and~3. In $\mathcal{G}(B_1,B_4)$, the two top vertices in the rightmost column are already pathwise connected, as are the two bottom vertices in its rightmost column. So appending $B_6$ to $\mathcal{G}(B_1,B_4)$ results in two short cycles as seen in Figure~\ref{BL}(c). The pathwise connected vertex pairs in the rightmost column of $\mathcal{G}(B_2,B_4)$ do not match up with those in the left column of $B_6$, so no short cycle is produced; instead, the graph  $\mathcal{G}(B_2,B_4, B_6)$ is one large cycle as shown in Figure~\ref{BL}(d). 

In general, we will see that for any compatible sequence $\mathbf{B}$, the pathwise connected pairs of vertices in the rightmost column of~$\mathcal{G}(\mathbf{B})$ determine whether appending a particular block produces a Hamiltonian cycle. We will also see that it is not necessary to know the entire sequence $\mathbf{B}$ in order to ascertain if appending a block to $\mathbf{B}$ creates an acyclic graph, a small cycle or a Hamiltonian cycle. To that end, we endow each block $B$ with an additional label that contains a list of pairs of endpoints of certain paths; these augmented blocks form the vertices of $D_{a,b}$. Edges in $D_{a,b}$ are defined in such a way that appending a mid block produces an acyclic graph and appending an end block produces one long cycle.

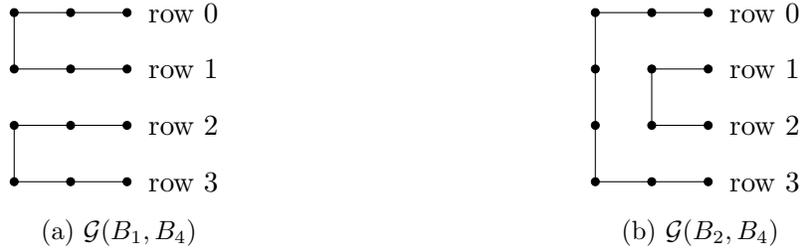
\begin{figure}[h!]
    \centering
    \begin{subfigure}[t]{0.5\textwidth}
        \centering
     \begin{tikzpicture}[scale=0.75]
     
\filldraw

(0,0) circle (2pt)
(0,1) circle (2pt)
(0,2) circle (2pt)
(0,3) circle (2pt)

(1,0) circle (2pt)
(1,1) circle (2pt)
(1,2) circle (2pt)
(1,3) circle (2pt)

(2,0) circle (2pt)
(2,1) circle (2pt)
(2,2) circle (2pt)
(2,3) circle (2pt);

\draw

(2,2)--(0,2)--(0,3)--(2,3)
(2,1)--(0,1)--(0,0)--(2,0);

\draw
(3,0) node{row 3}
(3,1) node{row 2}
(3,2) node{row 1}
(3,3) node{row 0};

\end{tikzpicture}
\caption{ $\mathcal{G}(B_1,B_4)$}
\label{hiya-ee}
\end{subfigure}%
~ 
\begin{subfigure}[t]{0.5\textwidth}
\centering
\begin{tikzpicture}[scale=0.75]
\filldraw

(0,0) circle (2pt)
(0,1) circle (2pt)
(0,2) circle (2pt)
(0,3) circle (2pt)

(1,0) circle (2pt)
(1,1) circle (2pt)
(1,2) circle (2pt)
(1,3) circle (2pt)

(2,0) circle (2pt)
(2,1) circle (2pt)
(2,2) circle (2pt)
(2,3) circle (2pt);

\draw

(2,0)--(0,0)--(0,3)--(2,3)
(2,2)--(1,2)--(1,1)--(2,1);

\draw 
(3,0) node{row 3}
(3,1) node{row 2}
(3,2) node{row 1}
(3,3) node{row 0};

\end{tikzpicture}
    \caption{$\mathcal{G}(B_2,B_4)$}
\end{subfigure}
\caption{Two sequences of blocks from Figure~\ref{14blocks}, with row labels.}
\label{14gluedgraphs}
\end{figure}

\begin{defn} \label{L(H)}
Let $H$ be a graph with maximum vertex degree~2 that is isomorphic to a spanning subgraph of $G_{a,b}(n)$ for some $n > b$. Define the set $\mathcal{L}(H)$ to consist of all unordered pairs $\{i,j\}$ with $0 \le i, j \le b-1$ such that  $H$ contains a path whose end vertices are located in its rightmost column at rows $i$ and $j$.
\end{defn}

For example, for the blocks $B_1, B_2$ in Figure~\ref{14blocks} and the two graphs in Figure~\ref{14gluedgraphs}, we have $\mathcal{L}(B_1) = \mathcal{L}(\mathcal{G}(B_1,B_4)) = \big \{ \{0,1\}, \{2,3\} \big \}$ and $\mathcal{L}(B_2) = \mathcal{L}(\mathcal{G}(B_2,B_4)) = \big \{ \{0,3\},\{1,2\} \big \}$. 

\begin{defn} \label{B^L}
Let $B$ be a mid or end block and $L$ any set of unordered pairs $\{i,j\}$ with $0 \leq i, j \leq b-1$. Define $B^L$ to be the graph obtained from $B$ by adding an edge joining the $i$-th and $j$-th vertices in the left column of $B$ whenever $\{i,j\} \in L$.
\end{defn}

Figure~\ref{BL} provides examples of graphs $B^L$ with $L = \mathcal{L}(\mathcal{G}(\mathbf{B}))$ for four compatible sequences $\mathbf{B}$ of blocks from Figure~\ref{14blocks}. 

\begin{figure}[h!]
\centering
 \begin{subfigure}[t]{0.44\textwidth}
    \centering
    \begin{tikzpicture}[scale=0.75]
     
\def\u{5}

\filldraw

(0,0) circle (2pt)
(0,1) circle (2pt)
(0,2) circle (2pt)
(0,3) circle (2pt)

(1,0) circle (2pt)
(1,1) circle (2pt)
(1,2) circle (2pt)
(1,3) circle (2pt)

(2,0) circle (2pt)
(2,1) circle (2pt)
(2,2) circle (2pt)
(2,3) circle (2pt);

\draw

(2,0)--(0,0)--(0,1)--(2,1)--(2,2)--(0,2)--(0,3)--(2,3);
 
\filldraw

(\u,0) circle (2pt)
(\u,1) circle (2pt)
(\u,2) circle (2pt)
(\u,3) circle (2pt)

(\u+1,0) circle (2pt)
(\u+1,1) circle (2pt)
(\u+1,2) circle (2pt)
(\u+1,3) circle (2pt);

\draw 

(\u+1,0)--(\u+0,0)--(\u+0,1)--(\u+1,1)--(\u+1,2)--(\u+0,2)--(\u+0,3)--(\u+1,3);
\end{tikzpicture}
\caption{$\mathcal{G}(B_1,B_5)$ and $B_5^{\mathcal{L}(B_1)}$} \label{B1B5}
\end{subfigure}
~
 \begin{subfigure}[t]{0.45\textwidth}
        \centering
\begin{tikzpicture}[scale=0.75]

\def\u{5}

\filldraw
(0,0) circle (2pt)
(0,1) circle (2pt)
(0,2) circle (2pt)
(0,3) circle (2pt)

(1,0) circle (2pt)
(1,1) circle (2pt)
(1,2) circle (2pt)
(1,3) circle (2pt)

(2,0) circle (2pt)
(2,1) circle (2pt)
(2,2) circle (2pt)
(2,3) circle (2pt);

\draw 

(2,0)--(0,0)--(0,3)--(2,3)
(2,2)--(1,2)--(1,1)--(2,1)--(2,2);

\filldraw

(\u,0) circle (2pt)
(\u,1) circle (2pt)
(\u,2) circle (2pt)
(\u,3) circle (2pt)

(\u+1,0) circle (2pt)
(\u+1,1) circle (2pt)
(\u+1,2) circle (2pt)
(\u+1,3) circle (2pt);

\draw

(\u+0,3) arc (130:230:2)
(\u+1,0)--(\u+0,0)
(\u+1,3)--(\u+0,3)
(\u+1,1)--(\u+0,1)--(\u+0,2)--(\u+1,2)--(\u+1,1);

\end{tikzpicture}
\caption{$\mathcal{G}(B_2,B_5)$ and $B_5^{\mathcal{L}(B_2)}$} \label{B2B5}
\end{subfigure}

\bigskip

\begin{subfigure}[b]{0.45\textwidth}
       
\centering
\begin{tikzpicture}[scale=0.75]
     
\def\u{5}

\filldraw

(0,0) circle (2pt)
(0,1) circle (2pt)
(0,2) circle (2pt)
(0,3) circle (2pt)

(1,0) circle (2pt)
(1,1) circle (2pt)
(1,2) circle (2pt)
(1,3) circle (2pt)

(2,0) circle (2pt)
(2,1) circle (2pt)
(2,2) circle (2pt)
(2,3) circle (2pt)

(3,0) circle (2pt)
(3,1) circle (2pt)
(3,2) circle (2pt)
(3,3) circle (2pt);

\draw

(3,0)--(0,0)--(0,1)--(3,1)--(3,0)
(3,2)--(0,2)--(0,3)--(3,3)--(3,2);

\filldraw
(\u,0) circle (2pt)
(\u,1) circle (2pt)
(\u,2) circle (2pt)
(\u,3) circle (2pt)

(\u+1,0) circle (2pt)
(\u+1,1) circle (2pt)
(\u+1,2) circle (2pt)
(\u+1,3) circle (2pt);

\draw 

(\u+1,0)--(\u+0,0)--(\u+0,1)--(\u+1,1)--(\u+1,0)
(\u+1,2)--(\u+0,2)--(\u+0,3)--(\u+1,3)--(\u+1,2);

\end{tikzpicture}
\caption{$\mathcal{G}(B_1,B_4, B_6)$ and $B_6^{\mathcal{L}(B_1, B_4)}$} \label{B1B4B6}
\end{subfigure}
~
\begin{subfigure}[b]{0.45\textwidth}
\centering
\begin{tikzpicture}[scale=0.75]

\def\w{5}

\filldraw
(0,0) circle (2pt)
(0,1) circle (2pt)
(0,2) circle (2pt)
(0,3) circle (2pt)

(1,0) circle (2pt)
(1,1) circle (2pt)
(1,2) circle (2pt)
(1,3) circle (2pt)

(2,0) circle (2pt)
(2,1) circle (2pt)
(2,2) circle (2pt)
(2,3) circle (2pt)

(3,0) circle (2pt)
(3,1) circle (2pt)
(3,2) circle (2pt)
(3,3) circle (2pt);

\draw 

(3,0)--(0,0)--(0,3)--(3,3)--(3,2)--(1,2)--(1,1)--(3,1)--(3,0);

\filldraw
(\w,0) circle (2pt)
(\w,1) circle (2pt)
(\w,2) circle (2pt)
(\w,3) circle (2pt)

(\w+1,0) circle (2pt)
(\w+1,1) circle (2pt)
(\w+1,2) circle (2pt)
(\w+1,3) circle (2pt);

\draw

(\w+0,3) arc (130:230:2)
(\w,3)--(\w+1,3)--(\w+1,2)--(\w,2)--(\w,1)--(\w+1,1)--(\w+1,0)--(\w,0);

\end{tikzpicture}
\caption{$\mathcal{G}(B_2,B_4, B_6)$ and $B_6^{\mathcal{L}(B_2, B_4)}$} \label{B2B4B6}
\end{subfigure}
\caption{Graphs $\mathcal{G}(\mathbf{B},B)$ and $B^{\mathcal{L}(\mathbf{B})}$ using blocks from Figure~\ref{14blocks}.}
\label{BL}
\end{figure}
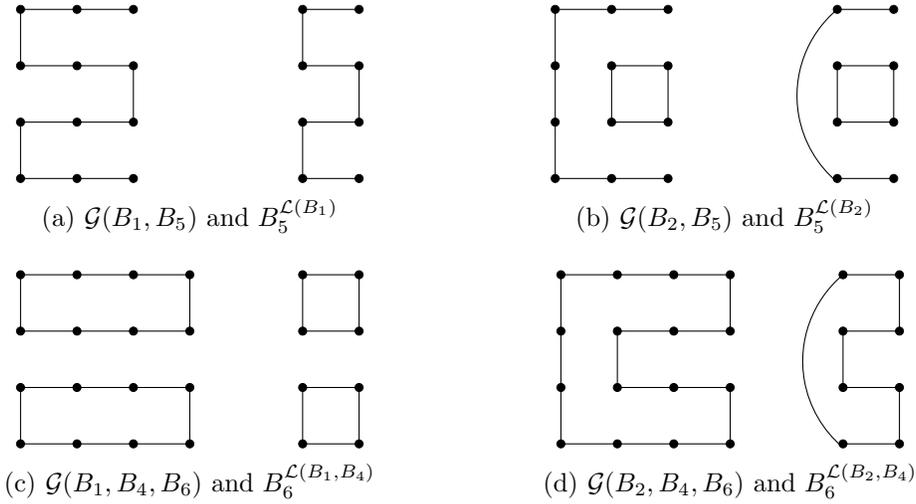

The motivation behind Definitions~\ref{L(H)} and~\ref{B^L} is illustrated in Figure~\ref{BL} which shows examples of graphs $\mathcal{G}(\mathbf{B},B)$ where $\mathbf{B}$ is one of four compatible sequences of blocks from Figure~\ref{14blocks} and $B$ is a block that is compatible with $\mathbf{B}$. For each of these graphs, the figure also depicts the associated graph $B^L$ where $L = \mathcal{L}(\mathcal{G}(\mathbf{B}))$. Note that every~$B^L$ has maximum  degree~2. For any two vertices $x,y$ in the left column of $B$, appending~$B$ to $\mathbf{B}$ potentially introduces a path through $\mathcal{G}(\mathbf{B})$ joining $x$ to $y$. In $B^L$, that path corresponds to an edge joining $x$ to $y$. Consequently, the path and cycle structure of $\mathcal{G}(\mathbf{B},B)$ is completely reflected in $B^L$. Paths with endpoints in the rightmost column of $\mathcal{G}(\mathbf{B},B)$ correspond to paths in $B^L$ with corresponding end points in its right column, so $\mathcal{L}(\mathcal{G}(\mathbf{B},B)) = \mathcal{L}(B^L)$. Cycles in $\mathcal{G}(\mathbf{B},B)$ correspond to cycles in $B^L$, and $\mathcal{G}(\mathbf{B},B)$ is a cycle if and only if $B^L$ is a cycle (along with possibly some isolated vertices). Consequently, it is sufficient to know~$B$ and $L$ only, rather than the entire sequence $\mathbf{B}$, to ascertain whether or not appending $B$ to $\mathbf{B}$ produces a Hamiltonian cycle. Here, inductively, $L$ is determined by the previous block, i.e.\ the last block in $\mathbf{B}$. The following lemma formulates this result more formally and provides a proof.

\begin{lemma}\label{useL}
Let $\mathbf{B}$ be a sequence of compatible blocks such that all but the first block in $\mathbf{B}$ are mid blocks. Let $B$ be a block such that $(\mathbf{B}, B)$ is compatible. Put $G = \mathcal{G}(\mathbf{B})$, $G' = \mathcal{G}(\mathbf{B}, B)$, $L = \mathcal{L}(G)$, and assume that $G$ is acyclic. Then the following conditions hold.

\begin{enumerate} \itemsep 0pt
    \item[(i)] $B^L$ has maximum degree $2$ and all the vertices in its left column have degree $0$ or $2$;
    \item[(ii)] $\mathcal{L}(G') = \mathcal{L}(B^L)$;
    \item[(iii)] The cycles in $G'$ are in one-to-one correspondence with those in $B^L$;
    \item[(iv)] Suppose $B$ is an end block. Then $G'$ is a cycle if and only if the graph obtained from $B^L$ by removing all isolated vertices is a cycle.
\end{enumerate}
\end{lemma}

\begin{proof} 
(i) We note that all the vertices in the right columns of $B$ and $B^L$ have same degree (in fact, the same adjacencies), which is at most~2 since $B$ is a block. Now let $x$ be any vertex in the left column of $B$, let $i$ be its row index, and let $y$ be the $i$-th vertex in the rightmost column of $G$. If $x$ has degree~0 or~2 in~$B$, then $y$ has degree~2 or~0 in $G$, since $(\mathbf{B}, B)$ is compatible.  Either way, $y$ is not the end point of any path in $G$. So no pair in~$L$ contains row index $i$, and no edge incident with $x$ is added to $B$ to obtain $B^L$. Thus, $x$ has the same degree in $B$ and $B^L$, namely~0 or~2. If $x$ has degree~1 in $B$, then $y$ has degree~1 in $G$ by compatibility and is hence the end point of a path in $G$. This is the only path in $G$ ending at $y$, since~$G$ has maximum degree~2. Thus, a single edge incident with $x$ is added to $B$ to form $B^L$. Consequently, $x$ has degree~2 in $B^L$. 
    
(ii) We first note that both $B^L$ and $G'$ have maximum degree~2, so~$\mathcal{L}(G')$ and~$\mathcal{L}(B^L)$ are defined. All paths in~$G$ and~$G'$ have endpoints in their respective rightmost columns since all their constituents blocks are compatible, and the same holds for~$B^L$ by part (i). Let $i, j \in \{ 0, 1, \ldots , b-1 \}$ be arbitrary, with $i \ne j$. We wish to prove that ${i,j} \in \mathcal{L}(G')$ if and only if ${i,j} \in \mathcal{L}(B^L)$. Any path connecting the $i$-th and $j$-th vertices in the rightmost column of $G$ corresponds uniquely to an edge in $B^L$ joining the $i$-th and $j$-th vertices in its first column. We prove that this correspondence extends to a one-to-one correspondence between the paths in $G'$ and $B^L$.
Consider $G$ and $B$ as subgraphs of $G'$ that share the second rightmost column of $G'$. Then every path in $G'$ is of the form 
\begin{equation}\label{pathform2}
P_1Q_1 \cdots  P_{k-1}Q_{k-1}P_k, 
\end{equation} 
where $P_1, \ldots , P_k$ are paths in $B$ and $Q_1, \ldots , Q_{k-1}$ are paths in $G$. Since $B$ is a subgraph of $B^L$, all the paths $P_1, \ldots , P_k$ are also paths in $B^L$, and the path given in \eqref{pathform2} corresponds to the path
\begin{equation}\label{pathform}
P_1 e_1\cdots P_{k-1} e_{k-1} P_k.
\end{equation} 
in $B^L$ where each edge $e_i$ corresponds to the path $Q_i$; see Figure~\ref{paths}. Conversely, every path in $B^L$ is of the form \eqref{pathform}, where $P_1, \ldots , P_k$ are paths in $B$ and $e_1, \ldots e_{k-1}$ are edges in the left column of $B^L$. Each edge $e_i$ in $B^L$ corresponds to a path $Q_i$ in $G$ with the same end points, yielding a path in $G'$ of the form \eqref{pathform2}.

\begin{figure}
\centering
\begin{tikzpicture}[scale=0.75]

\filldraw

(2,1) circle (2pt)
(2,2) circle (2pt)
(2,3) circle (2pt)
(2,4) circle (2pt)

(3,1) circle (2pt)
(3,2) circle (2pt)
(3,3) circle (2pt)
(3,4) circle (2pt);

\draw[thick, dashed]
(1.1,4) arc (-90:90:-0.5)
(1.1,2) arc (-90:90:-0.5);

\draw
(1,4)--(2,4)
(1,3)--(2,3)
(1,2)--(2,2)
(1,1)--(2,1)
(0.2,3.5) node{$Q_1$}
(0.2,1.5) node{$Q_2$}
(2.6,4.4) node{$P_1$}
(3.1,2.5) node{$P_2$}
(2.6,0.6) node{$P_3$}
(1.6,-0.5) node{$G'$};

\draw[very thick]
(2,4)--(3,4)
(2,3)--(3,3)--(2,2)
(2,1)--(3,1);

\def\u{6}

\filldraw

(\u+2,1) circle (2pt)
(\u+2,2) circle (2pt)
(\u+2,3) circle (2pt)
(\u+2,4) circle (2pt)

(\u+3,1) circle (2pt)
(\u+3,2) circle (2pt)
(\u+3,3) circle (2pt)
(\u+3,4) circle (2pt);

\draw
(\u+2,1) arc (230:130:0.65)
(\u+2,3) arc (230:130:0.65)

(\u+1.3,3.5) node{$e_{Q_1}$}
(\u+1.3,1.5) node{$e_{Q_2}$}
(\u+2.6,4.4) node{$P_1$}
(\u+3.1,2.5) node{$P_2$}
(\u+2.6,0.6) node{$P_3$}
(\u+2.5,-0.5) node{$B^L$};

\draw[very thick]
(\u+2,4)--(\u+3,4)
(\u+2,3)--(\u+3,3)
(\u+2,2)--(\u+3,3)
(\u+2,1)--(\u+3,1);

\end{tikzpicture}
\caption{Corresponding paths in $G'$ and $B^L$.} \label{paths}
\end{figure}
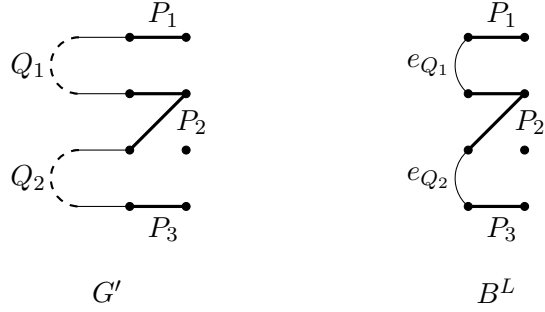

(iii) Since $G$ is acyclic, every cycle in $G'$ must contain a vertex in its rightmost column and is hence of the form~\eqref{pathform2}, where the initial vertex of $P_1$ and the terminal vertex of $P_k$ are identical and this vertex belongs to the rightmost column of $G'$. Part~(ii) now yields the desired result.

(iv) Suppose $G'$ is a cycle. By part (iii), $G'$ corresponds to a unique cycle $C$ in $B^L$, i.e. $C$ is the only cycle in $B^L$. Since $B$ is an end block, every vertex in the right column of $B^L$ has degree 2. By part (i), all the vertices in $B^L \setminus C$ have degree 0 and are hence isolated. 

Conversely, suppose $B^L$ consists of a cycle and zero or more isolated vertices. Since $B$ is an end block, $G'$ is 2-regular. By part (iii), $G'$ contains a unique cycle and is therefore itself a cycle. 
\end{proof}

We now have all the ingredients to construct $D_{a,b}$. This directed graph will contain four types of vertices:

\begin{enumerate} \itemsep 0pt
    \item A distinguished start vertex $S$;
    \item Vertices of the form $(B,L)$ where $B$ is any start block and $L = \mathcal{L}(B)$;
    \item Vertices of the form $(B,L)$ where $B$ is any mid block and $L$ is any set of zero or more pairs $\{ i, j \}$ with $0 \le i, j \le b-1$;
    \item End vertices $B$ where $B$ is an end block. 
\end{enumerate}
Once again, this collection of vertices is finite, since there are finitely many blocks and finitely many sets $L$. The edge set of $D_{a,b}$ is given as follows:

\begin{enumerate} \itemsep 0pt
    \item For any vertex of the form $(B, L)$, place an edge $\big (S,(B,L) \big )$ of weight $2b$ whenever $B$ is a start block, $L = \mathcal{L}(B)$ and $B$ is acyclic.
    \item For any two vertices of the form $(B,L), (\hat{B},\hat{L})$, place an edge $\big ( (B,L), (\hat{B},\hat{L}) \big)$ of weight $b$ whenever $\hat{B}$ is a mid block, $(B, \hat{B})$ is compatible, $\mathcal{L}(\hat{B}^L) = \hat{L}$ and $\hat{B}^L$ is acyclic. 
    \item For any two vertices of the form $(B,L), \hat{B}$, place an edge $\big ( (B,L), \hat{B} \big )$ of weight $s$, where $s$ is the length of the right column of $\hat{B}$, whenever $(B, \hat{B})$ is compatible and the graph obtained from $\hat{B}^L$ by removing all its isolated vertices is a cycle.
\end{enumerate}

Note that, implicitly, $B$ in step 2 is a start or mid block and $\hat{B}$ in step 3 is an end block.

Figure~\ref{D14} shows the subdigraph of $D_{1,4}$ induced by the vertices $S$ and $B^L$ where $B$ is any block of~Figure~\ref{14blocks} and $L_1 = \big \{\{0,1\},\{2,3\} \big \}$, $L_2 = \big \{\{0,3\},\{1,2\} \big \}$, and $L_3 = \big \{\{0,3\} \big \}$.

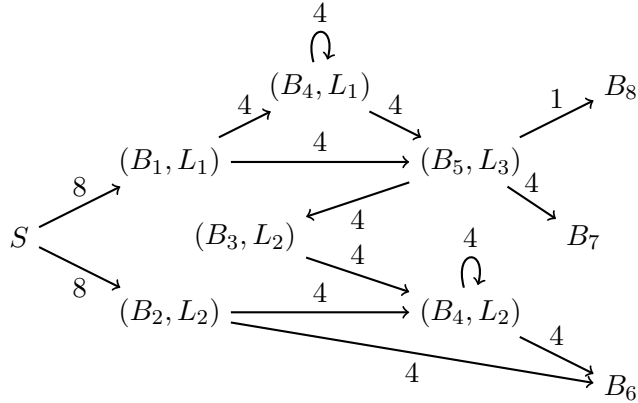
\begin{figure}[h!]
\begin{center}
\begin{tikzpicture}[scale=1]

\node (1) at (-1,0) {$S$};
\node (2) at (1,-1) {$(B_2,L_2)$};
\node (3) at (1,1) {$(B_1,L_1)$};
\node (4) at (3,2) {$(B_4,L_1)$};
\node (5) at (5,1) {$(B_5,L_3)$};
\node (6) at (5,-1) {$(B_4,L_2)$};
\node (7) at (2,0) {$(B_3,L_2)$};
\node (8) at (7,2) {$B_8$};
\node (9) at (6.5,0) {$B_7$};
\node (10) at (7,-2) {$B_6$};

\path[thick,->]
(1) edge node[below] {$8$} (2)
(1) edge node[above] {$8$} (3)
(2) edge node[below] {$4$} (10)
(2) edge node[above] {$4$} (6)
(3) edge node[above] {$4$} (4)
(3) edge node[above] {$4$} (5)
(4) edge node[above] {$4$} (5)
(5) edge node[below] {$4$} (7)
(7) edge node[above] {$4$} (6)
(5) edge node[above] {$1$} (8)
(5) edge node[above] {$4$} (9)
(6) edge node[above] {$4$} (10)
(4) edge[loop above] node {$4$} (4)
(6) edge[loop above] node {$4$} (6);

\end{tikzpicture}
\end{center}
\caption{A subdigraph of $D_{1,4}$}
\label{D14}
\end{figure}

\begin{theorem} \label{mainthm}
For any $n > 2b$, the Hamiltonian cycles of $G_{a,b}(n)$ are in one-to-one correspondence with the walks $W$ in $D_{a,b}$ from $S$ to any end vertex such that $S(W) = n$, with $S(W)$ as defined in \eqref{weightsum}.
\end{theorem}

\begin{proof} 
As in Proposition~\ref{2regular}, the proof is constructive and gives the explicit correspondence asserted in the theorem. 

Let $C$ be a Hamiltonian cycle in $G_{a,b}$, and let $q,r$ be as given in \eqref{ndiv}. By Lemma~\ref{decomp22}, there exists a unique compatible sequence $(B_1, B_2, \ldots , B_q)$ of blocks such that $B_1$ is a start block, $B_2, \ldots , B_{q-1}$ are mid blocks and $B_q$ is an end block whose right column has length $r+1$. Note that $q \ge 2$ as $n > 2b$. Put $H_j = \mathcal{G}(B_1,\ldots,B_j))$ and $L_j = \mathcal{L}(H_j)$ for $1 \leq j \leq q-1$. We claim that
\begin{equation}\label{walkform} 
W: S \longrightarrow (B_1,L_1) \longrightarrow (B_2,L_2) \longrightarrow \ldots \longrightarrow (B_{q-1},L_{q-1}) \longrightarrow B_q
\end{equation}
is a walk in $D_{a,b}$ with $S(W) = n$. We first establish that $W$ is a walk in $D_{a,b}$. First, note $G_j$ is acyclic for $1 \leq j \leq q-1$ since it is isomorphic to a proper subgraph of the cycle $C$. Since $L_1 = \mathcal{L}(H_1) = \mathcal{L}(B_1)$, $D_{a,b}$ contains the edge $\big ( (S, (B_1,L_1) \big )$. Let $j \in \{ 2, 3, \ldots , q-2 \}$. Then $B_j$ is a mid block and $(B_{j-1}, B_j)$ is compatible. Part (ii) of~Lemma~\ref{useL} now yields $\mathcal{L}(B_j^{L_{j-1}}) = \mathcal{L}(H_j) = L_j$, and part (iii) implies that $B_j^{L_j}$ is acyclic. It follows that $D_{a,b}$ contains the edge $\big ( (B_j,L_j), (B_{j+1},L_{j+1}) \big )$. Similarly, since $H_q$ is isomorphic to the cycle $C$, the graph obtained from $B_q^{L_{q-1}}$ by removing all its isolated vertices is a cycle by Lemma~\ref{useL}~(iii). So $D_{a,b}$ contains the edge $(B_q,L_q)$. 

The first edge $\big ( (S, (B_1,L_1) \big )$ in $W$ has weight $2b$. Each edge $\big ( (B_j,L_j), (B_{j+1},L_{j+1}) \big )$, with $1 \le j \leq q-2$, has weight $b$, and the edge $\big ( (B_{q-1},L_{q-1}), B_q \big )$ has weight $r+1$. So $S(W) = 2b + (q-2)b + r + 1 = n$.

Conversely, let $W$ be a walk in $D_{a,b}$ from $S$ to some end vertex such that $S(W) = n$. Then $W$ is of the form \eqref{walkform} for suitable sets of pairs $L_1, \ldots , L_q$ and suitable blocks $B_1, \ldots , B_q$. The construction of $D_{a,b}$ imposes the following conditions on these blocks:
\begin{enumerate} \itemsep 0pt
    \item $B_1$ is a start block, $B_2, \ldots , B_{q-1}$ are mid blocks and $B_q$ is an end block;
    \item $(B_j, B_{j+1})$ is compatible for $1 \leq j \le q$;
    \item $L_1 = \mathcal{L}(B_1)$ and $L_j = \mathcal{L}(B_j^{L_{j-1}})$ for $1 \leq j \leq q-1$;
    \item $B_1$ is acyclic and $B_{j+1}^{L_j}$ is acyclic for $1 \leq j \leq q-2$;
    \item The graph obtained from $B_q^{L_{q-1}}$ by removing all its vertices is a cycle.
\end{enumerate}

Put $H_j = \mathcal{G}(B_1,\ldots,B_j)$ for $1 \leq j \leq q$. We prove that $L_j = \mathcal{L}(H_j)$ and $H_j$ is acyclic for $1 \leq j \leq q-1$. Certainly $L_1 = \mathcal{L}(B_1) = \mathcal{L}(H_1)$ and $H_1 = B_1$ is acyclic. Assume inductively that $L_j = \mathcal{L}(H_j)$ and $H_j$ is acyclic for some $j \in \{ 1, \ldots , q-2 \}$. By Lemma~\ref{useL}~(ii), we have $\mathcal{L}(H_{j+1}) = \mathcal{L}(B_{j+1}^{\mathcal{L}(H_j)}) = \mathcal{L}(B_{j+1}^{L_j}) = L_{j+1}$. Since $B_{j+1}^{L_j}$ is acyclic, $H_{j+1}$ is acyclic by Lemma~\ref{useL}(iii).

Since $H_{q-1}$ is acyclic and $L_{q-1} = \mathcal{L}(H_{q-1})$, Lemma~\ref{useL}~(iv)
shows that $H_q$ is a cycle. The edge weights in $W$ yield $n = qb + s$ where $s$ is the length of the right column of~$B_q$. By Lemma~\ref{usecompatible}, $H_q$ is isomorphic to a 2-regular subgraph of $G_{a,b}(n)$ on $n$ vertices. It follows that $H_q$ is isomorphic to a Hamiltonian cycle in $G_{a,b}(n)$. 

The correspondence $H_q \longleftrightarrow W$ is the desired bijection. 
\end{proof}

Our main result is now again an immediate consequence of Lemma \ref{recur}.

\begin{cor} \label{mainresult}
The number of $(a,b)$-necklaces of length~$n$ satisfies a linear homogeneous recurrence relation.
\end{cor}

The construction of Theorem~\ref{mainthm}, and hence the result of Corollary~\ref{mainresult}, can be extended to necklaces where more than two differences are allowed. If $A$ is a finite set of positive integers, then we can define an $A$-necklace to be a circular arrangement of $\{0,1,\ldots,n-1\}$ such that adjacent beads have an absolute difference in $A$. If $N_A(n)$ denotes the number of $A$-necklaces of length $n$, then an analogous argument to the reasoning in this section shows that $N_A(n)$ satisfies a linear homogeneous recurrence relation whose coefficients are integers that depend on $A$.



\section{Conclusion} \label{S:conclude}

The proof of our main existence result for $(a,b)$-necklaces (Theorem~\ref{exist}) crucially requires that $a$ and $b$ not be too close together; specifically $b \ge 2a$. This restriction can almost certainly be removed, as formulated in Conjecture~\ref{conje}, but our construction does not cover the case $b < 2a$ and a new approach for this scenario is needed. The fact that the count of $(a,b)$-necklaces satisfies a recurrence relation is no help here. Although Theorem~\ref{expest} establishes the existence of $(a,b)$-necklaces of arbitrary length, it only proves their existence for lengths that are multiples of $a+b$. In the cases where $(a,b)$-necklaces of all sufficiently large lengths are known to exist, Corollary \ref{mainresult} shows that their number is either bounded or grows exponentially in the length of the necklace. 

The dependence on $a$ and $b$ of the minimal degree of the recurrence relation for $N_{a,b}(n)$ seems unclear. The degree of the recurrence obtained via Corollary~\ref{mainresult} can be bounded in terms of the number of vertices in the digraph obtained by subdividing the edges of $D_{a,b}$ in Theorem~\ref{mainthm} in accordance with Lemma~\ref{recur}. Counting sets of row indices, blocks and edge subdivisions yields a crude upper bound of $(cb)^b$ for some explicitly computable positive constant $c$, i.e.\ super-exponential in $b$. This bound seems far from tight based on the explicit recursions we obtained. Moreover, the technique need not yield a recurrence relation of minimal degree. 

Corollary~\ref{mainresult} extends to any finite set of differences.  In particular, the construction in the proof of Theorem~\ref{mainthm} is completely direct and explicit in the sense that moving along any walk in $D_{a,b}$ from the start vertex $S$ to an end vertex grows a Hamiltonian cycle in $G_{a,b}(n)$ for some $n$. It may be possible to apply the general technique to other counting problem of a similar flavour.

\section*{Acknowledgement}

The first and third author mourn the loss of their co-author Richard Guy who passed away on March 9, 2020, at the impressive age of 103. Richard was a mathematical giant, a passionate educator, a generous philanthropist and an avid mountaineer. To us, he was also a valued colleague, mentor and friend.


\vspace{1.5cm}

\begin{multicols}{2}
\small
\noindent {\sc Department of Mathematics  \\ The University of British Columbia \\Room 121, 1984 Mathematics Road \\ Vancouver, BC \\ Canada V6T 1Z2} \\
\url{epwhite@math.ubc.ca}\\

\columnbreak

\noindent {\sc \mbox{Department of Mathematics and Statistics} \\
University of Calgary \\
2500 University Drive NW \\
Calgary, AB \\
Canada T2N 3Z4} \\
\url{rscheidl@ucalgary.ca}

\end{multicols}

\vspace{1.5cm}

\appendix
\section{Appendix --- Numerical Data} \label{S:comput}

We generated tabulation data for $N_{a,b}(n)$ with the aid of a depth first search algorithm for constructing and counting $(a,b)$-necklaces, written in Python and run on a laptop. The algorithm searches the tree of sequences consisting of elements in the set $\{-b,-a,a,b\}$. A copy of our code, along with additional data files, are available from the first author upon request.

Tables~\ref{table1} and~\ref{table2} list the values  of $N_{a,b}(n)$ for $1 \leq a < b \leq 10$ and $n \leq 40$. To corroborate the correctness of our data, we verified that in each case,
\begin{itemize} \itemsep 0pt
    \item there is a unique $(a,b)$-necklace of length $a+b$,
    \item there is an $(a,b)$-necklace of length $3a+b$ when $2a \leq b$,
    \item there are no $(a,b)$-necklaces of odd length when $ab$ is odd and 
    \item the table entries agree with Theorem~\ref{expest} and the results of Section \ref{S:count}.
\end{itemize}

None of the sequences~$N_{a,b}(n)$ for which we obtained a meaningful amount of data appear in OEIS, with the exception of the pairs $(a,b) = (1,2), (1,3)$ and $(2,3)$. Using a number wall~\cite[pp. 85-89]{CG} (also known as a quotient-difference table), we checked that none of these counts $N_{a,b}(n)$ satisfy a linear recurrence whose degree is less than the number of terms computed, again with the exception of $(a,b) = (1,2), (1,3), (2,3), (1,4)$.

\begin{table}[ht] 
\begin{center}
\resizebox{17cm}{!}{
 \begin{tabular}{|c | c c c c c c c c c c c c c  c|} 
 \hline
 \backslashbox{$n$}{$a,b$} & 1,3 & 1,4 &1,5 & 1,6 & 1,7 & 1,8 & 1,9 & 1,10 & 2,3 & 2,5 & 2,7 & 2,9 & 3,4 & 3,5  \\ [0.5ex] 
 \hline\hline
 4 &1 & 0 & 0 & 0 & 0 & 0 & 0 & 0 & 0 & 0 & 0 & 0  & 0 & 0\\
  \hline
 5 &0 & 1 & 0 & 0 & 0 & 0 & 0 & 0 & 1 & 0 & 0 & 0  & 0 & 0 \\
  \hline
  6 &2 & 0 & 1 & 0 & 0 & 0 & 0 & 0 & 0 & 0 & 0 & 0  & 0 & 0\\
  \hline
    7 &0 & 1& 0 & 1 & 0 & 0 & 0 & 0 & 0 & 1 & 0 & 0 & 1 & 0 \\
  \hline
    8 &3 & 1& 1& 0 & 1 & 0 & 0 & 0 & 0 & 0 & 0 & 0 & 0 & 1 \\
  \hline
    9 &0 & 1& 0 & 1 & 0 & 1 & 0 & 0 & 0 & 0 & 1 & 0 & 0 & 0\\
  \hline
    10 &5 & 3& 2 & 0 & 1 & 0 & 1 & 0 & 1 & 0 & 0 & 0 & 0 & 0 \\
  \hline
    11 &0 & 2 & 0 & 1 & 0 & 1 & 0 & 1 & 1 & 1 & 0 & 1 & 0 & 0\\
  \hline
    12 &8 & 3& 5 & 1 & 1 & 0 & 1 & 0 & 1 & 0 & 0 & 0 & 0 & 0 \\
  \hline
    13 &0 & 6& 0 & 1 & 0 & 1 & 0 & 1 & 1 & 1 & 1 & 0 & 0  & 0\\
  \hline
    14 &13 & 5& 9 & 5 & 2 & 0 & 1 & 0 & 1 & 3 & 0 & 0 & 1 & 0 \\
  \hline
    15 &0 & 10& 0 & 2 & 0 & 1 & 0 & 1 & 2 & 0 & 0 & 1 & 0 & 0\\
  \hline
    16 &21 & 12& 18 & 7 & 7 & 1 & 1 & 0 & 3 & 1 & 0 & 0 & 0& 2\\
  \hline
    17 &0 & 14& 0 & 8 & 0& 1 & 0 & 1 & 4 & 5 & 0 & 0 & 0& 0 \\
  \hline
    18 &34 & 25& 34 & 8 & 13 & 7 & 2 & 0 & 5 & 5 & 5 & 0 & 0 & 2 \\
    \hline
    19 &0 & 27& 0& 24 & 0 & 2 & 0 & 1 & 6 & 2 & 2 & 1 & 0 & 0 \\
    \hline
    20 &55 & 40& 67 & 14 & 28 & 11 & 9 & 1 & 8 & 11 & 0 & 0 & 0  & 2\\
    \hline
    21 &0 & 57& 0& 46 & 0 & 8 & 0 & 1 & 11 & 18 & 0 & 1 & 2  & 0\\
    \hline
    22 &89 & 68& 131 & 35 & 64 & 16 & 17 & 9 & 15 & 8 & 8 & 7 & 1 & 5 \\
     \hline
    23 &0 & 104& 0 & 65 & 0 & 32 & 0 & 2 & 20 & 21 & 18 & 0 & 1 & 0 \\
     \hline
    24 &144 & 133& 251 & 105 & 124 & 20 & 36 & 15 & 26 & 48 & 4 & 0 & 0 & 14  \\
     \hline
    25 &0 & 177& 0 & 99 & 0 & 96 & 0 & 8 & 34 & 40 & 0 & 8 & 1 & 0 \\
     \hline
    26 &233 & 255& 493 & 248 & 231 & 36 & 92 & 24 & 45 & 48 & 5 & 16 & 1 & 28 \\
     \hline
    27 &0 & 324& 0 & 204 & 0 & 190 & 0 & 32 & 60 & 115 & 67 & 0 & 2 & 0 \\
     \hline
    28 &377 & 454& 956 & 437 & 495 & 87 & 232 & 36 & 80 & 145 & 53 & 4 & 8 & 46 \\
     \hline
    29 &0 & 617& 0 & 512 & 0 & 276 & 0 & 128 & 106 & 134 & 13 & 72 & 8 & 0 \\
     \hline
    30 &610 & 811& 1856 & 701 & 1061 & 276 & 454 & 48 & 140 & 272 & 8 & 36 & 10 & 93  \\
     \hline
    31 &0 & 1136& 0 & 1245 & 0 & 382 & 0 & 384 & 185 & 431 & 117 & 0 & 7& 0 \\
     \hline
    32 &987 & 1507& 3616 & 1280 & 2074 & 856 & 786 & 88 & 245 & 436 & 334 & 64 & 13 & 195\\
     \hline
    33 &0 & 2042& 0 & 2543 & 0 & 605 & 0 & 766 & 325 & 665 & 205 & 288 & 17 & 0 \\
     \hline
    34 &1597 & 2803& 7021 & 2784 & 4233 & 2136 & 1544 & 208 & 431 & 1161 & 80 & 35 & 28 & 399 \\
       \hline
    35 &0 & 3729& 0 & 4527 & 0 & 1275 & 0 & 1122 & 571 & 1402 & 176 & 41 & 52 & 0 \\
       \hline
    36 &2584 & 5109& 13656 & 6463 & 8914 & 3934 & 3530 & 650 & 756 & 1767 & 1026 & 489 & 69  & 764\\
       \hline
    37 &0 & 6904& 0 & 8106 & 0 & 3295 & 0 & 1568 & 1001 & 3020 & 1424 & 669 & 91 & 0\\
       \hline
    38 &4181 & 9290& 26551  & 13970 & 18237 & 6342 & 8056 & 2147 & 1326 & 4186 & 849 & 73 & 103  & 1508\\
       \hline
    39 &0 & 12692& 0 & 16162 & 0 & 9301 & 0 & 2188 & 1757 & 5071 & 629 & 581 & 148  & 0\\
       \hline
    40 &6765 & 17070& 51610 & 27115 & 36699 & 10282 & 15859 & 6669 & 2328 & 7848 & 2241 & 3140 & 204 & 3024  \\
   \hline
\end{tabular}}
\caption{Values of $N_{a,b}(n)$ for $1 \leq a < b \leq 10$ and $n \le 40$, part 1}
\label{table1}
\end{center}
\end{table}

\begin{table}[ht]
\begin{center}
\resizebox{17cm}{!}{
 \begin{tabular}{|c | c c c c c c c c c c c c c c c c|} 
 \hline
 \backslashbox{$n$}{$a,b$} & 3,7 & 3,8 &3,10 & 4,5 & 4,7 & 4,9 &5,6 & 5,7 & 5,8 & 5,9 & 6,7 & 7,8 & 7,9 & 7,10 & 8,9 & 9,10  \\ [0.5ex] 
 \hline\hline
 9 & 0 & 0 & 0 & 1 & 0 & 0 & 0 & 0 & 0 & 0 & 0 & 0 & 0  & 0 & 0 & 0 \\
   \hline 
 10 & 1 & 0 & 0 & 0 & 0 & 0 & 0 & 0 & 0 & 0 & 0 & 0 & 0  & 0 & 0 & 0 \\
  \hline 
   11 & 0 & 1 & 0 & 0 & 1 & 0 & 1 & 0 & 0 & 0 & 0 & 0 & 0  & 0 & 0 & 0 \\
  \hline 
  
   12 & 0 & 0 & 0 & 0 & 0 & 0 & 0 & 1 & 0 & 0 & 0 & 0 & 0  & 0 & 0 & 0 \\
  \hline 
   13 & 0 & 0 & 1 & 0 & 0 & 1 & 0 & 0 & 1 & 0 & 1 & 0 & 0  & 0 & 0 & 0 \\
  \hline 
   14 & 0 & 0 & 0 & 0 & 0 & 0 & 0 & 0 & 0 & 1 & 0 & 0 & 0  & 0 & 0 & 0 \\
  \hline 
   15 & 0 & 0 & 0 & 0 & 0 & 0 & 0 & 0 & 0 & 0 & 0 & 1 & 0  & 0 & 0 & 0 \\
  \hline 
   16 & 1 & 0 & 0 & 0 & 0 & 0 & 0 & 0 & 0 & 0 & 0 & 0 & 1  & 0 & 0 & 0 \\
  \hline 
   17 & 0 & 1 & 0 & 0 & 0 & 0 & 0 & 0 & 0 & 0 & 0 & 0 & 0  & 1 & 1 & 0 \\
  \hline 
   18 & 1 & 0 & 0 & 1 & 0 & 0 & 0 & 0 & 0 & 0 & 0 & 0 & 0  & 0 & 0 & 0 \\
  \hline 
  19 & 0 & 0 & 1 & 0 & 0 & 0 & 0 & 0 & 0 & 0 & 0 & 0 & 0  & 0 & 0 & 1 \\
  \hline 
  20 & 4 & 0 & 0 & 0 & 0 & 0 & 0 & 0 & 0 & 0 & 0 & 0 & 0  & 0 & 0 & 0 \\
  \hline 
  21 & 0 & 1 & 0 & 0 & 0 & 1 & 0 & 0 & 0 & 0 & 0 & 0 & 0  & 0 & 0 & 0 \\
  \hline 
  22 & 1 & 5 & 0 & 0 & 3 & 0 & 1 & 0& 0 & 0 & 0 & 0 & 0  & 0 & 0 & 0 \\
  \hline 
  23 & 0 & 0 & 0 & 0 & 0 & 1 & 0 & 0 & 0 & 0 & 0 & 0 & 0  & 0 & 0 & 0 \\
  \hline 
  24 & 10 & 0 & 0 & 0 & 1 & 0 & 0 & 2 & 0 & 0 & 0 & 0 & 0  & 0 & 0 & 0 \\
  \hline 
  25 & 0 & 0 & 0 & 0 & 2 & 0 & 0 & 0 & 0 & 0 & 0 & 0 & 0  & 0 & 0 & 0 \\
  \hline 
  26 & 19 & 0 & 7 & 0 & 0 & 5 & 0 & 0 & 3 & 0 & 1 & 0 & 0  & 0 & 0 & 0 \\
  \hline 
  27 & 0 & 15 & 2 & 2 & 1 & 0 & 0 & 0 & 0 & 0 & 0 & 0 & 0  & 0 & 0 & 0 \\
  \hline 
  28  & 34 & 14 & 0 & 0 & 4 & 1 & 0 & 0 & 0 & 4 & 0 & 0 & 0  & 0 & 0 & 0 \\
  \hline 
  29 & 0 & 0 & 0 & 0 & 1 & 0 & 0 & 0 & 1 & 0 & 0 & 0 & 0  & 0 & 0 & 0 \\
  \hline 
   30 & 86 & 0 & 0 & 0 & 9 & 0 & 0 & 0 & 0 & 1 & 0 & 1 & 0  & 0 & 0 & 0 \\
  \hline 
   31 & 0 & 2 & 0 & 0 & 8 & 13 & 0 & 0 & 0 & 0 & 0 & 0 & 0  & 0 & 0 & 0 \\
  \hline 
   32 & 115 & 33 & 20 & 0 & 5 & 4 & 0 & 0 & 0 & 3 & 0 & 0 & 2  & 0 & 0 & 0 \\
  \hline 
   33 & 0 & 113 & 54 & 0 & 48 & 5 & 2 & 0 & 0 & 0 & 0 & 0 & 0  & 0 & 0 & 0 \\
  \hline 
   34 & 366 & 8 & 0 & 0 & 20 & 44 & 0 & 0 & 0 & 1 & 0 & 0 & 0  & 3 & 1 & 0 \\
  \hline 
   35 & 0 & 0 & 0 & 0 & 32 & 2 & 0 & 0 & 2 & 0 & 0 & 0 & 0  & 0 & 0 & 0 \\
  \hline 
   36 & 615 & 6 & 0 & 8 & 115 & 68 & 0 & 14 & 5 & 8 & 0 & 0 & 0  & 0 & 0 & 0 \\
  \hline 
   37 & 0 & 93 & 0 & 2 & 43 & 18 & 0 & 0 & 0 & 0 & 0 & 0 & 0  & 0 & 0 & 0 \\
  \hline 
   38 & 1343 & 545 & 5 & 1 & 101 & 37 & 0 & 9 & 3 & 26 & 0 & 0 & 0  & 0 & 0 & 1 \\
  \hline 
   39 & 0 & 316 & 393 & 1 & 250 & 251 & 0 & 0 & 54 & 0 & 2 & 0 & 0  & 0 & 0 & 0 \\
  \hline 
   40 & 2841 & 0 & 134 & 0 & 133 & 34 & 0 & 5 & 0 & 40 & 0 & 0 & 0  & 0 & 0 & 0 \\
  \hline   
\end{tabular}}
\caption{Values of $N_{a,b}(n)$ for $1 \leq a < b \leq 10$ and $n \le 40$, part 2}
\label{table2}
\end{center}
\end{table}


\begin{thebibliography}{99}

\bibitem{JR} J. Ram\'irez Alfons\'in, \emph{The Diophantine Frobenius Problem}. Oxford University Press, 2005.

\bibitem{BG} E. R. Berlekamp and R. K. Guy, Fibonacci plays Billards. \url{arXiv:2002.03705 [math.HO]}. 

\bibitem{CG} J. H. Conway and R. K. Guy, \emph{The Book of Numbers}. Springer, New York 1996.


\bibitem{PF} P. Flajolet, K. Hatzis, S. Nikoletseas and P. Spirakis, On the robustness of interconnections in random graphs: a symbolic approach, \emph{Theoret. Comput. Sci.} \textbf{287} (2002), no. 2, 515--534


\bibitem{FS} P. Flajolet, R. Sedgewick, Analytic Combinatorics, \emph{Cambridge University Press}, (2009).


\bibitem{KD} Y. H. Kwong and D. G. Rogers, A matrix method for counting Hamiltonian cycles on grid graphs. \emph{European J. Combin.} \textbf{15} (1994), no. 3, 277--283

\bibitem{MO} Math Overflow, \url{https://mathoverflow.net/questions/199677/arranging-numbers-from-1-to-n-such-that-the-sum-of-every-two-adjacent-number/}.

\bibitem{MS} Mersenne Forum, \url{https://mersenneforum.org/showthread.php?p=477787}, January 2018.

\bibitem{MD} Mersenne Forum, \url{https://mersenneforum.org/showthread.php?t=22957}, January 2018.

\bibitem{P} V. H. Pettersson, Enumerating Hamiltonian Cycles.  \emph{Electron. J. Combin.}
\textbf{21} (2014), no. 4, Paper 4.7, 15 pp.

\bibitem{PP} Puzzle 311, \url{https://www.primepuzzles.net/puzzles/puzz_311.htm}. 


\bibitem{NS} N. J. A. Sloane (editor), Online Encyclopedia of Integer Sequences. Published electronically at \url{https://oeis.org/}.

\bibitem{SS} R. Stoyan and V.Strehl, Enumeration of Hamiltonian circuits in rectangular grids. \emph{J. Combin. Math. Combin. Comput.} \textbf{21} (1996), 109?127. 


\bibitem{G} M. Golin, Y.C. Leung, Y. Wang, Y. Xuerong, Counting Structures in Grid Graphs, Cyclinders and Tori Using Transfer Matrices: Survey and New Results, \emph{Proceedings of the Seventh Workshop on Algorithm Engineering and Experiments and the Second Workshop on Analytic Algorithmics and Combinatorics} (ALENEX/ANALCO 2005), 250--258, \url{http://hdl.handle.net/1783.1/14114}.

\bibitem{RS} R.P. Stanley. Rational Generating Functions. In Enumerative Combinatorics (202-292). Springer US. 

\bibitem{Y}  N. Yoshigahara, \emph{Puzzles 101: A Puzzlemaster's Challenge}, A K Peters, Natick MA, 2004. 


\end{thebibliography}
\end{document}